\documentclass[12pt,leqno]{amsart}
\usepackage{amssymb, amsmath}

\usepackage[left=2cm,right=2cm,top=2.5cm,bottom=2.5cm]{geometry}

\usepackage{caption} 
\usepackage[colorlinks,linkcolor=blue,pagebackref=true,citecolor=red]{hyperref}
\usepackage[all]{xy}
\usepackage{multicol}
\begin{document}

\theoremstyle{plain}
\newtheorem{thm}{Theorem}[section]
\newtheorem*{thm1}{Theorem 1}
\newtheorem*{thm2}{Theorem 2}
\newtheorem*{note*}{\it Note}
\newtheorem{lemma}[thm]{Lemma}
\newtheorem{lem}[thm]{Lemma}
\newtheorem{cor}[thm]{Corollary}
\newtheorem{prop}[thm]{Proposition}
\newtheorem{propose}[thm]{Proposition}
\newtheorem{variant}[thm]{Variant}
\theoremstyle{definition}
\newtheorem{notations}[thm]{Notations}
\newtheorem{rem}[thm]{Remark}
\newtheorem{rmk}[thm]{Remark}
\newtheorem{rmks}[thm]{Remarks}
\newtheorem{defn}[thm]{Definition}
\newtheorem{ex}[thm]{Example}
\newtheorem{exs}[thm]{Examples}
\newtheorem{claim}[thm]{Claim}
\newtheorem{ass}[thm]{Assumption}
\numberwithin{equation}{section}
\newcounter{elno}                
\def\points{\list
{\hss\llap{\upshape{(\roman{elno})}}}{\usecounter{elno}}} 
\let\endpoints=\endlist


\catcode`\@=11
%
%
\def\opn#1#2{\def#1{\mathop{\kern0pt\fam0#2}\nolimits}} 
\def\bold#1{{\bf #1}}%
\def\underrightarrow{\mathpalette\underrightarrow@}
\def\underrightarrow@#1#2{\vtop{\ialign{$##$\cr
 \hfil#1#2\hfil\cr\noalign{\nointerlineskip}%
 #1{-}\mkern-6mu\cleaders\hbox{$#1\mkern-2mu{-}\mkern-2mu$}\hfill
 \mkern-6mu{\to}\cr}}}
\let\underarrow\underrightarrow
\def\underleftarrow{\mathpalette\underleftarrow@}
\def\underleftarrow@#1#2{\vtop{\ialign{$##$\cr
 \hfil#1#2\hfil\cr\noalign{\nointerlineskip}#1{\leftarrow}\mkern-6mu
 \cleaders\hbox{$#1\mkern-2mu{-}\mkern-2mu$}\hfill
 \mkern-6mu{-}\cr}}}
%
%

%
\def\:{\colon}
\let\oldtilde=\tilde
\def\tilde#1{\mathchoice{\widetilde{#1}}{\widetilde{#1}}%
{\indextil{#1}}{\oldtilde{#1}}}
\def\indextil#1{\lower2pt\hbox{$\textstyle{\oldtilde{\raise2pt%
\hbox{$\scriptstyle{#1}$}}}$}}
\def\pnt{{\raise1.1pt\hbox{$\textstyle.$}}}
%

%
\let\amp@rs@nd@\relax
\newdimen\ex@\ex@.2326ex
\newdimen\bigaw@l
\newdimen\minaw@
\minaw@16.08739\ex@
\newdimen\minCDaw@
\minCDaw@2.5pc
\newif\ifCD@
\def\minCDarrowwidth#1{\minCDaw@#1}
\newenvironment{CD}{\@CD}{\@endCD}
\def\@CD{\def\A##1A##2A{\llap{$\vcenter{\hbox
 {$\scriptstyle##1$}}$}\Big\uparrow\rlap{$\vcenter{\hbox{%
$\scriptstyle##2$}}$}&&}%
\def\V##1V##2V{\llap{$\vcenter{\hbox
 {$\scriptstyle##1$}}$}\Big\downarrow\rlap{$\vcenter{\hbox{%
$\scriptstyle##2$}}$}&&}%
\def\={&\hskip.5em\mathrel
 {\vbox{\hrule width\minCDaw@\vskip3\ex@\hrule width
 \minCDaw@}}\hskip.5em&}%
\def\verteq{\Big\Vert&&}%
\def\noarr{&&}%
\def\vspace##1{\noalign{\vskip##1\relax}}\relax\let\amp@rs@nd@&\iffalse}\fi
 \CD@true\vcenter\bgroup\relax\let\\=\cr\iffalse}\fi\tabskip\z@skip\baselineskip20\ex@
 \lineskip3\ex@\lineskiplimit3\ex@\halign\bgroup
 &\hfill$\m@th##$\hfill\cr}
\def\@endCD{\cr\egroup\egroup}
%
\def\>#1>#2>{\amp@rs@nd@\setbox\z@\hbox{$\scriptstyle
 \;{#1}\;\;$}\setbox\@ne\hbox{$\scriptstyle\;{#2}\;\;$}\setbox\tw@
 \hbox{$#2$}\ifCD@
 \global\bigaw@\minCDaw@\else\global\bigaw@\minaw@\fi
 \ifdim\wd\z@>\bigaw@\global\bigaw@\wd\z@\fi
 \ifdim\wd\@ne>\bigaw@\global\bigaw@\wd\@ne\fi
 \ifCD@\hskip.5em\fi
 \ifdim\wd\tw@>\z@
 \mathrel{\mathop{\hbox to\bigaw@{\rightarrowfill}}\limits^{#1}_{#2}}\else
 \mathrel{\mathop{\hbox to\bigaw@{\rightarrowfill}}\limits^{#1}}\fi
 \ifCD@\hskip.5em\fi\amp@rs@nd@}
\def\<#1<#2<{\amp@rs@nd@\setbox\z@\hbox{$\scriptstyle
 \;\;{#1}\;$}\setbox\@ne\hbox{$\scriptstyle\;\;{#2}\;$}\setbox\tw@
 \hbox{$#2$}\ifCD@
 \global\bigaw@\minCDaw@\else\global\bigaw@\minaw@\fi
 \ifdim\wd\z@>\bigaw@\global\bigaw@\wd\z@\fi
 \ifdim\wd\@ne>\bigaw@\global\bigaw@\wd\@ne\fi
 \ifCD@\hskip.5em\fi
 \ifdim\wd\tw@>\z@
 \mathrel{\mathop{\hbox to\bigaw@{\leftarrowfill}}\limits^{#1}_{#2}}\else
 \mathrel{\mathop{\hbox to\bigaw@{\leftarrowfill}}\limits^{#1}}\fi
 \ifCD@\hskip.5em\fi\amp@rs@nd@}
%
%
\newenvironment{CDS}{\@CDS}{\@endCDS}
\def\@CDS{\def\A##1A##2A{\llap{$\vcenter{\hbox
 {$\scriptstyle##1$}}$}\Big\uparrow\rlap{$\vcenter{\hbox{%
$\scriptstyle##2$}}$}&}%
\def\V##1V##2V{\llap{$\vcenter{\hbox
 {$\scriptstyle##1$}}$}\Big\downarrow\rlap{$\vcenter{\hbox{%
$\scriptstyle##2$}}$}&}%
\def\={&\hskip.5em\mathrel
 {\vbox{\hrule width\minCDaw@\vskip3\ex@\hrule width
 \minCDaw@}}\hskip.5em&}
\def\verteq{\Big\Vert&}
\def\novarr{&}
\def\noharr{&&}
\def\SE##1E##2E{\slantedarrow(0,18)(4,-3){##1}{##2}&}
\def\SW##1W##2W{\slantedarrow(24,18)(-4,-3){##1}{##2}&}
\def\NE##1E##2E{\slantedarrow(0,0)(4,3){##1}{##2}&}
\def\NW##1W##2W{\slantedarrow(24,0)(-4,3){##1}{##2}&}
\def\slantedarrow(##1)(##2)##3##4{%
\thinlines\unitlength1pt\lower 6.5pt\hbox{\begin{picture}(24,18)%
\put(##1){\vector(##2){24}}%
\put(0,8){$\scriptstyle##3$}%
\put(20,8){$\scriptstyle##4$}%
\end{picture}}}
\def\vspace##1{\noalign{\vskip##1\relax}}\relax\let\amp@rs@nd@&\iffalse}\fi
 \CD@true\vcenter\bgroup\relax\let\\=\cr\iffalse}\fi\tabskip\z@skip\baselineskip20\ex@
 \lineskip3\ex@\lineskiplimit3\ex@\halign\bgroup
 &\hfill$\m@th##$\hfill\cr}
\def\@endCDS{\cr\egroup\egroup}
%
\newdimen\TriCDarrw@
\newif\ifTriV@
\newenvironment{TriCDV}{\@TriCDV}{\@endTriCD}
\newenvironment{TriCDA}{\@TriCDA}{\@endTriCD}
\def\@TriCDV{\TriV@true\def\TriCDpos@{6}\@TriCD}
\def\@TriCDA{\TriV@false\def\TriCDpos@{10}\@TriCD}
\def\@TriCD#1#2#3#4#5#6{%
\setbox0\hbox{$\ifTriV@#6\else#1\fi$}
\TriCDarrw@=\wd0 \advance\TriCDarrw@ 24pt
\advance\TriCDarrw@ -1em
\def\SE##1E##2E{\slantedarrow(0,18)(2,-3){##1}{##2}&}
\def\SW##1W##2W{\slantedarrow(12,18)(-2,-3){##1}{##2}&}
\def\NE##1E##2E{\slantedarrow(0,0)(2,3){##1}{##2}&}
\def\NW##1W##2W{\slantedarrow(12,0)(-2,3){##1}{##2}&}
\def\slantedarrow(##1)(##2)##3##4{\thinlines\unitlength1pt
\lower 6.5pt\hbox{\begin{picture}(12,18)%
\put(##1){\vector(##2){12}}%
\put(-4,\TriCDpos@){$\scriptstyle##3$}%
\put(12,\TriCDpos@){$\scriptstyle##4$}%
\end{picture}}}
\def\={\mathrel {\vbox{\hrule
   width\TriCDarrw@\vskip3\ex@\hrule width
   \TriCDarrw@}}}
\def\>##1>>{\setbox\z@\hbox{$\scriptstyle
 \;{##1}\;\;$}\global\bigaw@\TriCDarrw@
 \ifdim\wd\z@>\bigaw@\global\bigaw@\wd\z@\fi
 \hskip.5em
 \mathrel{\mathop{\hbox to \TriCDarrw@
{\rightarrowfill}}\limits^{##1}}
 \hskip.5em}
\def\<##1<<{\setbox\z@\hbox{$\scriptstyle
 \;{##1}\;\;$}\global\bigaw@\TriCDarrw@
 \ifdim\wd\z@>\bigaw@\global\bigaw@\wd\z@\fi
 \mathrel{\mathop{\hbox to\bigaw@{\leftarrowfill}}\limits^{##1}}
 }
 \CD@true\vcenter\bgroup\relax\let\\=\cr\iffalse}\fi
 \tabskip\z@skip\baselineskip20\ex@
 \lineskip3\ex@\lineskiplimit3\ex@
 \ifTriV@
 \halign\bgroup
 &\hfill$\m@th##$\hfill\cr
#1&\multispan3\hfill$#2$\hfill&#3\\
&#4&#5\\
&&#6\cr\egroup%
\else
 \halign\bgroup
 &\hfill$\m@th##$\hfill\cr
&&#1\\%
&#2&#3\\
#4&\multispan3\hfill$#5$\hfill&#6\cr\egroup
\fi}
\def\@endTriCD{\egroup} 

\newcommand{\mc}{\mathcal} 
\newcommand{\mb}{\mathbb} 
\newcommand{\surj}{\twoheadrightarrow} 
\newcommand{\inj}{\hookrightarrow} \newcommand{\zar}{{\rm zar}} 
\newcommand{\an}{{\rm an}} \newcommand{\red}{{\rm red}} 
\newcommand{\Rank}{{\rm rk}} \newcommand{\codim}{{\rm codim}} 
\newcommand{\rank}{{\rm rank}} \newcommand{\Ker}{{\rm Ker \ }} 
\newcommand{\Pic}{{\rm Pic}} \newcommand{\Div}{{\rm Div}} 
\newcommand{\Hom}{{\rm Hom}} \newcommand{\im}{{\rm im}} 
\newcommand{\Spec}{{\rm Spec \,}} \newcommand{\Sing}{{\rm Sing}} 
\newcommand{\sing}{{\rm sing}} \newcommand{\reg}{{\rm reg}} 
\newcommand{\Char}{{\rm char}} \newcommand{\Tr}{{\rm Tr}} 
\newcommand{\Gal}{{\rm Gal}} \newcommand{\Min}{{\rm Min \ }} 
\newcommand{\Max}{{\rm Max \ }} \newcommand{\Alb}{{\rm Alb}\,} 
\newcommand{\GL}{{\rm GL}\,} 
\newcommand{\ie}{{\it i.e.\/},\ } \newcommand{\niso}{\not\cong} 
\newcommand{\nin}{\not\in} 
\newcommand{\soplus}[1]{\stackrel{#1}{\oplus}} 
\newcommand{\oOplus}{\mbox{\fontsize{17.28}{21.6}\selectfont\( \oplus\)}}
\newcommand{\by}[1]{\stackrel{#1}{\rightarrow}} 
\newcommand{\longby}[1]{\stackrel{#1}{\longrightarrow}} 
\newcommand{\vlongby}[1]{\stackrel{#1}{\mbox{\large{$\longrightarrow$}}}} 
\newcommand{\ldownarrow}{\mbox{\Large{\Large{$\downarrow$}}}} 
\newcommand{\lsearrow}{\mbox{\Large{$\searrow$}}} 
\renewcommand{\d}{\stackrel{\mbox{\scriptsize{$\bullet$}}}{}} 
\newcommand{\dlog}{{\rm dlog}\,} 
\newcommand{\longto}{\longrightarrow} 
\newcommand{\vlongto}{\mbox{{\Large{$\longto$}}}} 
\newcommand{\limdir}[1]{{\displaystyle{\mathop{\rm lim}_{\buildrel\longrightarrow\over{#1}}}}\,} 
\newcommand{\liminv}[1]{{\displaystyle{\mathop{\rm lim}_{\buildrel\longleftarrow\over{#1}}}}\,} 
\newcommand{\norm}[1]{\mbox{$\parallel{#1}\parallel$}} 
\newcommand{\boxtensor}{{\Box\kern-9.03pt\raise1.42pt\hbox{$\times$}}} 
\newcommand{\into}{\hookrightarrow} \newcommand{\image}{{\rm image}\,} 
\newcommand{\Lie}{{\rm Lie}\,} 
\newcommand{\CM}{\rm CM}
\newcommand{\sext}{\mbox{${\mathcal E}xt\,$}} 
\newcommand{\shom}{\mbox{${\mathcal H}om\,$}} 
\newcommand{\coker}{{\rm coker}\,} 
\newcommand{\sm}{{\rm sm}} 
\newcommand{\tensor}{\otimes} 
\renewcommand{\iff}{\mbox{ $\Longleftrightarrow$ }} 
\newcommand{\supp}{{\rm supp}\,} 
\newcommand{\ext}[1]{\stackrel{#1}{\wedge}} 
\newcommand{\onto}{\mbox{$\,\>>>\hspace{-.5cm}\to\hspace{.15cm}$}} 
\newcommand{\propsubset} {\mbox{$\textstyle{ 
\subseteq_{\kern-5pt\raise-1pt\hbox{\mbox{\tiny{$/$}}}}}$}} 
\newcommand{\sB}{{\mathcal B}} \newcommand{\sC}{{\mathcal C}} 
\newcommand{\sD}{{\mathcal D}} \newcommand{\sE}{{\mathcal E}} 
\newcommand{\sF}{{\mathcal F}} \newcommand{\sG}{{\mathcal G}} 
\newcommand{\sH}{{\mathcal H}} \newcommand{\sI}{{\mathcal I}} 
\newcommand{\sJ}{{\mathcal J}} \newcommand{\sK}{{\mathcal K}} 
\newcommand{\sL}{{\mathcal L}} \newcommand{\sM}{{\mathcal M}} 
\newcommand{\sN}{{\mathcal N}} \newcommand{\sO}{{\mathcal O}} 
\newcommand{\sP}{{\mathcal P}} \newcommand{\sQ}{{\mathcal Q}} 
\newcommand{\sR}{{\mathcal R}} \newcommand{\sS}{{\mathcal S}} 
\newcommand{\sT}{{\mathcal T}} \newcommand{\sU}{{\mathcal U}} 
\newcommand{\sV}{{\mathcal V}} \newcommand{\sW}{{\mathcal W}} 
\newcommand{\sX}{{\mathcal X}} \newcommand{\sY}{{\mathcal Y}} 
\newcommand{\sZ}{{\mathcal Z}} \newcommand{\ccL}{\sL} 
 \newcommand{\A}{{\mathbb A}} \newcommand{\B}{{\mathbb 
B}} \newcommand{\C}{{\mathbb C}} \newcommand{\D}{{\mathbb D}} 
\newcommand{\E}{{\mathbb E}} \newcommand{\F}{{\mathbb F}} 
\newcommand{\G}{{\mathbb G}} \newcommand{\HH}{{\mathbb H}} 
\newcommand{\I}{{\mathbb I}} \newcommand{\J}{{\mathbb J}} 
\newcommand{\M}{{\mathbb M}} \newcommand{\N}{{\mathbb N}} 
\renewcommand{\P}{{\mathbb P}} \newcommand{\Q}{{\mathbb Q}} 

\newcommand{\R}{{\mathbb R}} \newcommand{\T}{{\mathbb T}} 
\newcommand{\U}{{\mathbb U}} \newcommand{\V}{{\mathbb V}} 
\newcommand{\W}{{\mathbb W}} \newcommand{\X}{{\mathbb X}} 
\newcommand{\Y}{{\mathbb Y}} \newcommand{\Z}{{\mathbb Z}} 

\title{Numerical characterizations for integral dependence of graded modules}
\author{Suprajo Das}
\address{Mathematical and Physical Sciences Division, School of Arts and Sciences, Ahmedabad University, Navrangpura, Ahmedabad, Gujarat 380009, India}
\email{suprajo.das@ahduni.edu.in}
\author{Sudeshna Roy}
\address{Department of Mathematics, Indian Institute of Technology Gandhinagar, Palaj, Gandhinagar, Gujarat 382055, India}
\email{sudeshnaroy.11@gmail.com; sudeshna.roy@iitgn.ac.in}
\author{Vijaylaxmi Trivedi}
\address{Department of Mathematics, University at Buffalo (SUNY),  Buffalo, NY 14260, USA}
\email{vijaylax@buffalo.edu}

\begin{abstract} 
In this paper we construct {\em adic}, 
{\em saturated} and  $\varepsilon$-density functions for a torsion free module in a graded setup.  Then 
we give some simple criteria for checking the integral dependence of two graded modules  $N\subseteq M$ in terms of various well-studied invariants.
\end{abstract}

\maketitle

\section{Introduction}

The main objective of this article is to
construct density functions and 
provide a new (necessary and sufficient) numerical criterion for detecting integral dependence of arbitrary graded modules defined over a standard graded Noetherian domain over a field.
This extends the earlier results 
in \cite{DRT25} and \cite{DRT26}, which dealt with  graded ideals.

The idea of characterizing integral dependence through numerical invariants, was initiated in the pioneering work of Rees \cite{Ree61}. Since then finding such criteria became an important task in both commutative algebra and singularity theory. For a nice detailed survey, one may refer to \cite[Chapter 11]{HS06}.
 There are also extensions of Rees’ theorem to the case of modules and algebras, see \cite{KT94,  SUV01, Cid24, CRPU24, PTUV20}.

 Here we consider  a finitely generated $\N$-graded torsion-free $A$-module $M$ of rank $e$, where $A$ is a standard graded domain of dimension $d\geq 2$ over a field $k$.  We fix 
a graded embedding
$M\longto F$, where $F$ is a free graded $A$-module of finite rank.

The Rees algebra of $M$ is defined as 
$$ A[Mt] = \oplus_nM^nt^n = \mbox{image}
({\rm Sym}_A(M)\longto 
{\rm Sym}_A(F)),$$
where ${\rm Sym}_A(L)$ denote the symmetric algebra of the $A$-module $L$ and 
$M^n $ denotes the image of $n^{th}$ symmetric power of $M$ in ${\rm Sym}^n_AF = F^n$. We note here that ${\rm Sym}_A(F)$ is a polynomial ring over $A$.
Further if $N\subseteq M$ is a graded submodule then  the Rees algebra of $N$ is defined as  
$$A[Nt] = \oplus_nN^nt^n = \mbox{image}~({\rm Sym}_A(N)\longto {\rm Sym}_A(F)),$$ 
in particular $N^n\subseteq M^n\subseteq F^n$.

 For a graded $A$-module $L$ let $d_L$ denote the maximum generating degree of a minimal set of homogeneous generators. 

We consider the natural 
 bigrading on the Rees algebra $A[Mt] =\oplus_{(m,n)\in\N^2}{(M^n)}_mt^n$ of the module $M$ as follows. Let  $A = k[x_1, \ldots, x_r]$, where $x_i's$ are degree $1$ elements. 
 If $M$ is generated by elements $m_1, \ldots, 
 m_s$ of degrees $d_1\leq  \ldots\leq  d_s$. 
 Then $A[Mt]$  is a finitely generated bigraded $A$-algebra 
generated by $m_1, \ldots, m_s$, where degree of $m_i$ in this bigraded algebra is 
 $(d_i, 1)$ and  bidgree of each $x_i$ is $(1,0)$.

 In this paper we prove the following.
 
 \vspace{5pt}

\noindent{\bf Theorem}~(A).~~(see~Theorem \ref{t2})\quad{\em 
Let $A$ and $M$ be  as above. Then
{\em adic} density function for $M$
$$f_{M}:\R\longto \R_{\geq 0}\quad\mbox{given by}\quad x\longto \lim_{n\to \infty}\frac{\ell_k\big((M^n)_{\lfloor xn\rfloor}\big)}{n^{d+e-2}/(d+e-1)!},$$
is a well defined function which is continous everywhere except possibly at  the point $x = d_1$. Moreover
$$f_{M}(x) = \begin{cases}
              0 & \text{if}\;\;x\in I(-\infty, d_{1}),\\
              {\bf p}_1(x) & \text{if}\;\;x\in I({d_1}, {d_2}],\\
              {\bf p}_j(x) & \text{if}\;\;x\in I[{d_j}, d_{j+1}]\;\;\text{where}\;\;2\leq j<l,\\
              {\bf p}_j(x) & \text{if}\;\;x\in I[{d_l},\infty),
             \end{cases}$$
where ${\bf p}_j(x)\in \Q[x]$ is a polynomial of degree $\leq d-1$ and ${\bf p}_l(x)$ is a nonzero polynomial of degree $d-1$.}

\vspace{5pt}

\noindent{\bf Theorem}~(B).~~(see Theorem~\ref{t3})\quad{\em 
 Let the {\em saturated density function} of $M$ be  the function given 
$${\tilde f}_M:\R\longto \R_{\geq 0} \quad\mbox{given by}\quad 
x \to  \lim_{n\to \infty}g_n(x),$$
 where 
  ${g}_n:\R\longto \R_{\geq 0}$ is the function given by
$${g}_n(x) = \frac{\ell_k((\tilde {M^n})_{\lfloor xn\rfloor})}{n^{d+e-2}/(d+e-1)! },\;\;\mbox{where}\;\;
{\tilde {M^n}} = (M^n:_{F^n}{\bf m}^{\infty}).$$
Then ${\tilde f_M}$ is  a well defined function which is  continuous and 
 supported on the interval $I[-c_0, \infty)$, for some constant $c_0\geq 0$.}

Along the way to establish the existence of the adic and the saturation density functions, we observe that 
there is a polynomial $P_M(x)$ for  of degree $d-1$ 
such that 
$f_M(x) = {\tilde f}_M(x)= P_M(x)$ $x\geq d_M$,
where 
$$P_M(x) = \frac{e_{d-1}(A(Mt])}{(d-1)!(e-1)!}x^{d-1}+
 \frac{e_{d-2}(A(Mt])}{(d-2)!(e)!}x^{d-2}+
 \cdots+\frac{e_0(A[Mt])}{(d+e-2)!},$$
 and where  the integer $e_i(A[Mt])$ is known (in \cite{HT03}) as the $i^{th}$ mixed multiplicity of $A[Mt]$.
 
 \vspace{5pt}

\noindent{\bf Theorem}~(C).~~(see Theorem~\ref{l1})\quad{\em 
Let  the {\em epsilon density ($\varepsilon$-density) function} of $M$ be the function $$f_{{\varepsilon}(M)}:\R\longto \R_{\geq 0} \quad\mbox{given by}\quad 
x \to \lim_{n\to \infty}f_n(\varepsilon)(x),$$
where $f_n(\varepsilon):\R \longto \R_{\geq 0}$ is defined as
$$ f_n(\varepsilon)(x) = g_n(x)-f_n(x) = \frac{\ell_k(({\tilde M^n}/M^n)_{\lfloor xn\rfloor})}{n^{d+e-2}/(d+e-1)!}. $$

Then it is well defined 
function, where  
$f_{{\varepsilon}(M)}(x)= {\tilde f_M}(x)-f_M(x)$ 
 is continuous on $\R\setminus \{d_1\}$ and is supported on $I[-c_0, d_M]$, for some constant 
$c_0\geq 0$. In particular
$$\varepsilon(M) = \int_{\R}f_{\varepsilon(M)}(x)dx  = \lim_{n\to \infty}\sum_{m=-c_0n}^{d_Mn}\frac{\ell_k(({\tilde M^n}/M^n)_m)}{n^{d+e-1}/(d+e-1)!}.$$}

 The existence of the limit 
 $\varepsilon(M) = \lim_{n\to \infty}\ell_k({\tilde M^n}/M^n)/n^{d+e-1}$ was established in \cite[Corollary~2]{Cut11} in  more general context namely  when $A$ is 
 local domain of depth $\geq 2$ and  is essentially finite type over a field $k$.
 The existence of $\varepsilon(M)$ as limsup was given in \cite{UV11}. Here in 
 Theorem~(C) we give another proof of existence as the limit, although in the graded setup.

 Now for a given module $M$ as above and a graded submodule $N\subseteq M$ 
 we denote the $(c, 1)$-diagonal subalgebras of $A[Nt]$ and $A[Mt]$ as   
 $$A[Nt]_{\Delta_{(c, 1)}} = \oplus_n(N_c)^n_{cn}t^n \quad\mbox{and}\quad
 A[Mt]_{\Delta_{(c, 1)}} = \oplus_n(M_c)^n_{cn}t^n.$$ 
 We know that if $c> \max\{d_M, d_N\}$ then
$A[Nt]_{\Delta_{(c, 1)}}$ and 
$A[Mt]_{\Delta_{(c, 1)}}$ are standard graded 
$k$-algebras.
 We prove the following criteria for integral dependence of $N\subseteq M$.

 \vspace{5pt}
 
\noindent{\bf Theorem}~(D)~~(see Theorem~\ref{ID}) {\em Let  $N\subseteq 
M\subseteq F$ as above, where $\rank\;M = \rank\;N = e$ and $\dim\;A=d \geq 2$. Further let $S = A[y]$, ${\mathsf N} = N\tensor_AS$ and 
${\mathsf M} = M\tensor_AS$. We denote $d_{N, M} = \max\{d_M, d_N\}$.
Then the following statements are equivalent.
\begin{enumerate}
 \item $N$ is a reduction of $M$.
  \item $f_M\equiv f_N$ and 
 ${\tilde f}_M \equiv {\tilde f}_N$.
 \item $\varepsilon(M) = \varepsilon(N)$ and $e(A[{M}t]_{\Delta(c,1)}) =  
e(A[{N}t]_{\Delta(c,1)})$ for 
some integer $c > d_{N, M}$.
\item  
$\varepsilon({\mathsf M}_cS) = 
 \varepsilon({\mathsf N}_cS)$ for 
some integer $c > d_{N, M}$.
\item $e(S[{\mathsf M}t]_{\Delta(c,1)}) =  
e(S[{\mathsf N}t]_{\Delta(c,1)})$ for 
some integer $c > d_{N, M}$.

 \item 
 $e_i(S[{\mathsf M}t]) = e_i(S[{\mathsf N}t])$ for all $0\leq i \leq d$.
 \end{enumerate}}

Here the strategy for the proofs  is similar but not exactly same as the one we used for ideals in \cite{DRT25} and \cite{DRT26}.
For example, when $F$ is a free module of rank $\geq 2$ and  
invariants of $F^n$ vary as  $n$ varies in $\N$ unlike in the ideal case where 
$A^n = A$ for all $n$. Also, for an ideal $I$ all the powers $I^n$ are  in the same module, namely $A$ and are related.
One of the new crucial ingredients,
is the linear regularity bound for powers $M^n$  of $M$ as in  \cite[Corollary~2]{Kodi00}.

\section{preliminaries}

Throughout the paper we use  the following notations. Here  $k$ is a field and $A =\oplus_{m\geq 0}A_m$ is an  standard graded domain of Krull dimension $d\geq 2$
over a field  $A_0=k$.  The unique homogeneous maximal ideal $\oplus_{m\geq 1}A_m$ of $A$ is denoted by ${\bf m}$ and 
$M$ denotes a finitely generated $\N$-graded torsion free $A$-module.
To define 
a  Rees algebra of $M$ we consider  
 an embedding of graded map of $M\longto F$, where $F$ is graded $A$-module.
 There can be several such maps.
 However 
since $M$ is a torsion free module over an integral domain the Rees algebra will be  independent of the 
choice of the embedding (see \cite{HS06} for details).
 We  describe two of such embedddings.

 \begin{enumerate} 
\item Versal map in the sense of \cite{EHU03}:
Define $$\Hom^*_A(M,A) = \sum_n{\underline\Hom}_A(M, A)_n,$$
where 
${\underline\Hom}_A(M,A)_n = \{f\in \Hom_A(M,A)\quad\mbox{such that}\quad f(M_i)\subset 
A_{n+i}\quad\mbox{for all}\quad  i\}$.
This gives a $\Z$-grading to the $A$-module $\Hom^*_A(M, A)$ and therefore to  the dual module $M^* = 
\Hom_A(M, A) =\Hom^*_A(M, A)$,  where the second equality holds  because $M$ is a finitely 
generated $A$-module.
Now  we  choose a free $A$-module $F$ of finite rank and a sequence of graded maps
$$M\longto M^{**}\longto F,$$ where the last map is the dual of a graded surjective map $F\longto M^*$.

\item Suppose  $\rank_A\;M =e$. Then from any given set of  homogeneous generators of $M$ we can choose elements $m_1, \ldots, m_e$ of nonnegative degrees, say,  $f_1, \ldots, f_e$ and a homogeneous element $s\in A$ of degree, say, $c_0$ such that 
$sM \subseteq E = \oplus_{i=1}^eAm_i\subseteq  
M$. This implies $E$ is free $A$-module and there is an injective
graded  map $M(-c_0)\longto E$ and therefore 
$$M\longto E(c_0) = \oplus_{i=1}^eA(c_0-f_i).$$

We recall the following example. Let $I = (x^2, y^3)$ an ideal in the polynomial ring
$A = k[x,y]$ over a field $k$. Let $M = I(2)$, a shift of the ideal $I$. 
Then there is the canonical graded map
$M \longto A(2)$, where $M$ is $\N$-graded and $A(2)$ has nonzero components in degrees $-1$ and $-2$. 
\end{enumerate}

\begin{defn}\label{d1}
For a finitely generated $\N$-graded torsion-free $A$-module $M$ we fix 
a graded embedding
$M\longto F$, where $F$ is a free graded $A$-module of finite rank.
This gives a graded $A$-algebra maps
${\rm Sym}_A(M)\longto 
{\rm Sym}_A(F)),$
where   ${\rm Sym}_A(L)$ denotes the symmetric algebra of the $A$-module $L$. 
Now
the Rees algebra of $M$ is defined as 
$$ A[Mt] = \oplus_nM^nt^n = \mbox{image}
({\rm Sym}_A(M)\longto 
{\rm Sym}_A(F)),$$
where  $M^n $ denotes the image of $n^{th}$ symmetric power of $M$ in ${\rm Sym}^n_AF = F^n$. We note here that $Sym_A(F)$ is a polynomial ring over $A$.

Further if $N\subseteq M$ is a graded submodule then  the Rees algebra of $N$ is defined as  
$$A[Nt] = \oplus_nN^nt^n = \mbox{image}~({\rm Sym}_A(N)\longto {\rm Sym}_A(F)),$$ 
in particular $N^n\subseteq M^n\subseteq F^n$.
\end{defn}

\begin{notations}\label{n1}
\begin{enumerate}\item Let  $d$ denote $\dim\;A$ where we assume  $d\geq 2$ and let $e$ denote   $\rank\;M$.
\item
Let $m_1, m_2, \ldots, m_s$ be a minimal set of   homogeneous generators of degrees  $d_1, \ldots, d_s$, reply. Without loss of generality we can assume that   $1\leq d_1< d_2<\cdots <d_l$ are all the  distinct integers among this set of degrees. 
\item For a graded $A$-module $L$ let $d_L$ denote the maximum generating degree of a minimal set of homogeneous generators. In particular for the module $M$   as above  $d_M =  d_l$.
\item 
We consider the natural 
 bigrading on the Rees algebra $A[Mt] =\oplus_{(m,n)\in\N^2}{(M^n)}_mt^n$ of the module $M$ as follows. Let  $A = k[x_1, \ldots, x_r]$, where $x_i's$ are degree $1$ elements. Then $A[Mt]$  is a finitely generated bigraded $A$-algebra 
generated by $m_1, \ldots, m_s$, where degree of $m_i$ is in this bigraded algebra is 
 $(d_i, 1)$ and  bidgree of each $x_i$ is $(1,0)$.
 
\end{enumerate}
\end{notations}

We recall the following results from 
\cite[Lemmas~2.9, 2.10. 2.11 and 3.3, Corollary~3.4 ]{DRT25}, stated here as Lemmas~\ref{cgl1}, \ref{cgl2}, \ref{cgl3}, \ref{qp} and Corollary~\ref{HT}.
Lemmas~\ref{cgl1}, \ref{cgl2} and \ref{cgl3}  are the results in euclidean  geometry. The Lemma~\ref{qp} and Corollary~\ref{HT} were proved in \cite{DRT25}
for any finitely generated bigraded $A$-algebra, so  they hold in particular in the present context. Also 
 the results in \cite{DRT25} were stated and proved in the context when the 
 bigraded algebras generated by elements in degrees $(1,0)$ and $\{(d_i,e_i)\mid d_i, e_i\geq 1\}_i$, wheras here we assume in addition that 
 all  
$e_i = 1$. We also recall here relavant Notations.

\subsection{The length function of a Noetherian bigraded algebra over \texorpdfstring{$k$}{k}}\label{ss3}

Let $A = \oOplus_{n\geq 0}A_n$  be  a standard graded algebra over  a field $A_0 = k$ with 
homogeneous maximal ideal $\mathfrak{m} = \oOplus_{n\geq 1}A_n$.
Therefore, we can write $A = k[x_1, \ldots, x_r]$, where $x_i$'s are degree $1$  elements.

Let $A[Mt] = \oplus_{(m,n)\in \N^2}(M^n)_m$ be a finitely generated bigraded
$A$-algebra generated by elements of degrees 
$\{(d_i, 1)\in \N^2\mid 1\leq i \leq s\}$, where bidegree of $x_i$ is $(1,0)$. Then 
 $A[Mt]$ is represented as a graded quotient of 
the bigraded polynomial ring 
$S=k[X_1,\ldots,X_r,Y_1,\ldots,Y_s]$, where 
$X_i$ maps to $x_i$ and $Y_i$ map to a degree $(d_i, 1)$ generator of $A[Mt]$. Therefore there exists \cite[Proposition $8.18$]{MS05} a bigraded minimal
free resolution of finite length  
$$0\to \mathop{\oOplus}_{j=1}^{\eta_t}
S(-a_{tj},-b_{tj})^{\beta_{tj}} \to \cdots \to \mathop{\oOplus}_{j=1}^{\eta_1}
S(-a_{1j},-b_{1j})^{\beta_{1j}} \to S \to A[Mt]\to 0$$ 
of $A[Mt]$ as an $S$-module. Hence the bigraded Hilbert series of $A[Mt]$ is given by
$$\begin{array}{lcl}
\sum\limits_{(m,n)\in\mathbb{N}^2}\ell_k((M^n)_m) x^my^n & = &
\sum\limits_{i=0}^t (-1)^i\Big(\sum\limits_{j=1}^{\eta_i}
\beta_{ij}\big(\sum\limits_{(m,n)\in\mathbb{N}^2}
\ell(S_{m-a_{ij},n-b_{ij}}) x^my^n\big)\Big)\\\
& = & \dfrac{p(x,y)}{(1-x)^r(1-x^{d_1}y)\cdots (1-x^{d_s}y)},
\end{array}$$
where $p(x,y) = \sum_{i=0}^t (-1)^i
\big(\sum_{j=1}^{\eta_i}\beta_{ij}x^{a_{ij}}y^{b_{ij}}\big)
\in\mathbb{Z}[x,y]$.
 We rewrite
$$p(x,y)  = \sum_{(\alpha, \beta)\in N(I)}c_{\alpha,
\beta}x^{\alpha}y^{\beta}\in \Z[x, y],\quad\mbox{where}\quad
 N(I) = \{(\alpha, \beta)\in \N\times \N\mid c_{\alpha, \beta}\neq 0\}.$$
This gives
\begin{equation}\label{pe}\sum\limits_{(m,n)\in\mathbb{N}^2}\ell_k((M^n)_m) x^my^n =
\sum_{(\alpha,\beta)\in N(I)}c_{\alpha,\beta}~\ell_k(S_{m-\alpha, n-\beta})x^my^n.\end{equation}

\begin{notations}\label{chambers}
Following Notations~\ref{n1} the set $\{d_1, d_2, \ldots, d_s\}$ denotes a set of  degrees of homogeneous generators, where $d_1< d_2<\ldots <d_l$ are the distinct elements in this set.  
\begin{enumerate}

\item Let $\mathfrak{C}_j\subset \R^2$ be the cone generated by the vectors $(d_j, 1)$ and $(d_{j+1}, 1)$ for $1\leq j \leq l-1$ and $\mathfrak{C}_l$ be the cone generated by the vectors $(d_l, 1)$ and 
$(1, 0)$.

\item We define the {\em restricted cone} corresponding to
$\mathfrak{C}_j$ as
$${\mathfrak R}C_j=
\cap_{(\alpha, \beta)\in N(I)} 
{\mathfrak C}_j+(\alpha, \beta).$$ Then 
${\mathfrak R}C_j=
{\mathfrak C}_j+(\lambda_j, \beta_j)$, 
for some $(\lambda_j, \beta_j)\in \Q^2_{\geq 0}$.
Note that since the constant term of the polynomial $p(x,y)$ is  the length of bigraded component $(M^0)_0 = A_0 =k$ which is nonzero, we have  $(0,0)\in N(I)$.
In particular, ${\mathfrak R}C_j \subseteq  {\mathfrak C_{j}}$. Further
$\lambda_j-{\beta_jd_j}$ and $\beta_jd_{j+1}-\lambda_j$ are nonnegative rational numbers, for all
$1\leq i\leq l$.
\item To avoid confusion with a tuple $(r_1, r_2)\in \R^2$ we will denote the open interval $(r_1, r_2)$ by $I(r_1, r_2)$ and the closed interval $[r_1, r_2]$ by $I[r_1, r_2]$.

\item We can partition
$\N\times \N = ({\mathfrak C}_0\cap \N\times \N)\cup
 ({\mathfrak C}_1\cap \N \times \N)\cup
 \cdots \cup({\mathfrak C}_l\cap \N\times\N),$
 where ${\mathfrak C}_0$ is the cone generated by the vectors $(0, 1)$ and $(d_1, 1)$. It follows from definition that if $1\leq j<l$ then
 $$(m,n)\in {\mathfrak R}C_j\iff 
 m/n \in I\big[d_j+\tfrac{(\lambda_j-{\beta_jd_j})}{n}, \quad d_{j+1}
 -\tfrac{(\beta_jd_{j+1}-\lambda_j)}{n}\big]$$
and $(m,n)\in {\mathfrak R}C_l\iff  
m/n \in I\big[d_l+\tfrac{(\lambda_l-\beta_ld_l)}{n}, \quad \infty).$

\end{enumerate}
\end{notations}

\begin{lemma}\label{cgl1}
Let $x\in I(d_j, d_{j+1})$ be a fixed real number. Then for any given finite set $\mathcal{S}\subset \mathbb{\R}^2$ there exists an integer $n_{\mathcal{S}}$ such that $\left(\lfloor xn\rfloor + \alpha, n+\beta\right)\in \mathfrak{R}C_{j}$ for every $n\geq n_{\mathcal{S}}$ and all $(\alpha, \beta)\in \mathcal{S}$.
\end{lemma}

\begin{lemma}\label{cgl2}
Let the hypotheses be as in Notations~\ref{chambers}. Let $\mathcal{S}\subset \R$ be a finite set and $2\leq i\leq l$. Then for all $a\in \mathcal{S}$ and all $n\gg 0$,
\begin{enumerate}
\item there exists $m_0\in e_1\N_{>0}$ such that $(a, 0)+n(d_i, 1) + m_0(d_1, 1) \in {\mathfrak R}C_{i-1}$ and
\item there exists $m_1\in \N_{>0}$ such that $(a, 0)+n(d_i, 1) - m_1(1,0)
\in {\mathfrak R}C_{i-1}$,
\end{enumerate}
\end{lemma}

\begin{lemma}\label{cgl3}
Let $\{d_j\}_{1\leq j\leq 1}$ be the set of elements as in Notations $\ref{chambers}$. Let $\sS\subset \R$ be a finite set. Then the following holds.
\begin{enumerate}
\item If $d_1 < d_{i} < d_{l+1}$, {\em i.e.}, $2\leq i\leq l$ then there exists $m_0\in e_1\N_{>0}$ such that
$$(a, 0)+ n(d_i, 1) - m_0(d_1, 1)\in {\mathfrak R}C_{i}\quad\mbox{for every}\;\; a\in \sS\;\;\mbox{and all}\;\; n\gg 0.$$
\item If $d_1\leq d_i < d_{l+1}$, {\em i.e.}, $1\leq i \leq l$ then there exists $m_1\in \N_{>0}$ such that
$$(a, 0) + n(d_i, 1) + m_1(1, 0)\in {\mathfrak R}C_{i},\quad\mbox{for every}\;\; a\in \sS\;\;\mbox{and all}\;\; n\gg 0.$$
\end{enumerate} 
\end{lemma}

\begin{lemma}\label{qp}Following Notations $\ref{chambers}$, we have that for each cone $\mathfrak{C}_{j_0}$ there is a quasi-polynomial $QP_{j_0}(X, Y)$ which is periodic of period $h\in \N_{>0}$ in both variables $X$ and $Y$ such that $$(m,n)\in {\mathfrak{R}}C_{j_0} \implies \ell(R_{m,n})= QP_{j_0}(m,n).$$
\end{lemma}

The following Corollary was also follows from  \cite[Theorem~1.1]{HT03}.

\begin{cor}\label{HT} Let $R$ be  the bigraded ring as in Notations $\ref{chambers}$.
Then there exists a polynomial $P(X, Y)\in \Q[X,Y]$ such that
$$m\geq d_{l}n + n_0 \implies \ell_k((M^n)_m) = P(m,n),$$
$\mbox{where} ~n_0 =  \lceil{\lambda_l} - d_l \beta_l\rceil$ with $(\lambda_l, \beta_l) \in \Q^2_{\geq 0}$, associated to ${\mathfrak{R}{C_l}}= {\mathfrak{C_l}}+(\lambda_l, \beta_l)$.
\end{cor}

\section{Adic density function} 

Following Theorem was proved in
\cite[Theorem~3.10]{DRT25}
for a Noetherian filtration  $\{I_n\}_n$ of homogeneous ideals in $A$. Here we need to modify some of the arguments as some of the inqualities, such as  $I_{n+1}\subseteq I_n$,
no longer hold for the case of family of modules $\{M^n\}_n$.

\begin{thm}\label{t1} Let $A[Mt] =   
\bigoplus_{(m,n)\in\N^2}{(M^n)}_mt^n$
 be the  bigraded $A$-algebra as in Notations~\ref{n1}.
 Then for a given cone  ${\mathfrak C_{j_0}}$ as in Notations~\ref{chambers}  there exist a homogeneous polynomial $P_{j_0}(X, Y)$ and a quasi-polynomial $Q_{j_0}(X,Y)$ in $\Q[X, Y]$ with $\deg~Q_{j_0}(X,Y) < \deg~P_{j_0}(X, Y)$ such that
 \begin{equation}\label{et1}
 (m,n)\in {\mathfrak{R}C_{j_0}} \cap \N^2 \implies  
\ell_k\big((M^n)_m\big) = P_{j_0}(m,n) + Q_{j_0}(m, n).\end{equation}
Further $\deg P_{j_0}(X,Y) \leq d+e-2$.
\end{thm}
\begin{proof} By Lemma~\ref{qp}, for the cone ${\mathfrak C}_{j_0}$, there is a quasi-polynomial $QP_{j_0}(X, Y)$
which is periodic of period $h$ in both variables $X$ and $Y$, where $h\geq 1$ is an integer, such that
$$(m, n) \in {\mathfrak R}C_{j_0} \implies 
\ell_k((M^n)_m)  =  QP_{j_0}(m, n).$$
Thus there exists a set
$$\sS_{j_0} =  \{P_{i,j}(X, Y)\in \mathbb{Q}[X,Y] \mid 0\leq i, j<h\}.$$ 
Further for all $i, j\in \N$ we identify  $P_{i, j}(X, Y)$ with a polynomial in $\sS_{j_0}$ by defining
$P_{i, j}(X, Y) = P_{i{\rm mod}\;h, j
{\rm mod}\;h}(X, Y)$.

Now we fix a point $x_0\in I(d_{j_0}, d_{j_0+1})$. Then for all $n\gg 0$ we have $(\lfloor x_0n\rfloor, n)\in {\mathfrak R}C_{j_0}$.
Now 
given
 $P_{0,0}(X,Y)$ and given $(\lfloor x_0n\rfloor, n)\in {\mathfrak R}C_{j_0}$ we  choose
 $\alpha_0, \beta_0 \in \{1, \ldots, h\}$  such that 
 \begin{equation}\label{alpha}
  \ell_k\big((M^{n+\beta_0})_{\lfloor x_0n\rfloor + \alpha_0}\big) =
P_{0, 0}(\lfloor x_0n\rfloor + \alpha_0, n+\beta_0).\end{equation}
With this choice of  $\alpha_0, \beta_0$ we have, for all $i, j\geq 0$
$$\ell_k\big((M^{n+\beta_0+j})_{\lfloor x_0n\rfloor + \alpha_0+i}\big)  =
P_{i, j}(\lfloor x_0n\rfloor + \alpha_0+i, n+\beta_0+j).$$

Let $\deg P_{i,j}(X, Y) = r_{i,j}$ and set $r_0 = \max\{r_{i,j}\}_{i,j}$. We consider the finite set
$$\sS=\{(\alpha, \beta)\in \N^2\mid 0\leq \alpha\leq (r_0+2+hd_1)h,\quad 0\leq \beta\leq 2h\},$$
where $d_1\geq 1$ denotes the integer as in Notations~\ref{n1}.
We can further choose  $n_{x_0}>0$ such that for all $n\geq n_{x_0}$
\begin{equation}\label{n_x}
 (\lfloor x_0n\rfloor + \alpha, n+\beta) \in {\mathfrak R}C_{j_0}\cap 
\N^2\quad\mbox{for each}\quad
(\alpha, \beta)\in \sS.\end{equation}
In particular
$$(\lfloor x_0n\rfloor +\alpha_0+i+i_1h, n+\beta_0+j)\in 
{\mathfrak R}C_{j_0}\cap 
\N^2\quad\mbox{for all}\quad 0\leq i_1\leq r_0,~~~ 1\leq i, j \leq h$$

and for a choice of $\alpha_0, \beta_0$ as in (\ref{alpha})  
\begin{equation}\label{e2_poly}
\ell_k\big((M^{n+\beta_0+j})_{\lfloor x_0n\rfloor +
\alpha_0+i+i_1h}\big) =
P_{i, j}(\lfloor x_0n\rfloor + \alpha_0+i+i_1h, n+\beta_0+j).\end{equation}

To prove the first assertion of the theorem, we need to show that the top degree terms of the polynomials  
$P_{i,j}(X, Y)$ is independent of the choice of $(i,j)$. For $h=1$ there is nothing to prove so henceforth we assume that $h\geq 2$.
Moreover since the choice of $P_{0,0}$ is random, it is enough to prove the following.

\noindent{\bf Claim}. \begin{enumerate}
\item[$(i)$] The top degree terms of  $P_{0,0}(X, Y)$ = the top degree terms of $P_{j,0}(X, Y)$, for all $1\leq j < h$.
\item[$(ii)$] The top degree terms of $P_{0,0}(X, Y)$ = the top degree terms of $P_{0,j}(X, Y)$, for all $1\leq j < h$.
\end{enumerate}

 \vspace{5pt}

\noindent{\emph {Proof of the Claim $(i)$}}.\quad Let ${\tilde r} = \max\{r_{0,0}, r_{j,0}\}$ (so ${\tilde r}\leq r_0$). Then we can express
$$P_{0,0}(X,Y) = \sum_{l=0}^{\tilde r}\lambda_l^0X^lY^{{\tilde r}-l}+R_{0,0}(X,Y)\quad\mbox{and}\quad   
P_{j,0}(X,Y) = \sum_{l=0}^{\tilde r}\lambda_l^jX^lY^{{\tilde r}-l}+R_{j,0}(X,Y),$$
where degrees of $R_{0,0}(X,Y)$ and $R_{j,0}(X,Y)$ are 
strictly less than ${\tilde r}$.

Therefore we can write 
$$\tfrac{P_{0,0}(X,Y)}{Y^{\tilde r}} =
\sum_{l=0}^{\tilde r}\lambda^0_l(\tfrac{X}{Y})^l + \sum_{i=1}^{\tilde r}
\tfrac{R^0_{{\tilde r}-i}({X}/{Y})}{Y^i},\quad\quad
\tfrac{P_{j,0}(X,Y)}{Y^{\tilde r}} =
\sum_{l=0}^{\tilde r}\lambda^j_l(\tfrac{X}{Y})^l + \sum_{i=1}^{\tilde r}
\tfrac{R^j_{{\tilde r}-i}({X}/{Y})}{Y^i},$$
where $R^0_i({X}/{Y})$ and $R^j_i({X}/{Y})$ are polynomials in $\Q[{X}/{Y}]$ of degrees $\leq i$. Note that if $r_{j,0} < {\tilde r}$ then all $\lambda^j_l = 0$ for all $l$. Similar statement holds if $r_{0,0} < {\tilde r}$.

If the claim does not hold then
$\sum_{l=0}^{\tilde r}(\lambda^0_l-\lambda^j_l)(\tfrac{X}{Y})^l$ is a nonzero polynomial of degree $\leq {\tilde r}$.
Now  $\{\frac{\lfloor xn\rfloor + \alpha_0+ih}{n+\beta_0}\mid 
0\leq i\leq {\tilde r}\}$ 
is a set of ${\tilde r}+1$ distinct points.
Therefore for some
$$(m_1, n_1) = (\lfloor x_0n\rfloor + \alpha_0+i_1h, n+\beta_0)~\mbox{we have}~ \sum_{l=0}^{{\tilde r}}(\lambda^0_l-\lambda^j_l)
\left(\tfrac{m_1}{n_1}\right)^l \neq 0.$$

Moreover $\{(m_1, n_1), (m_1h^i, n_1h^i)\mid i\geq 1\}
\cup \{(m_1+j, n_1), (m_1+h, n_1) \}
\subseteq {\mathfrak R}C_{j_0}$ implies 
\begin{multline*}
\{(m_1h^i+m_1, n_1h^i+n_1)\mid i\geq 1\} \cup
\{(m_1h^i+m_1+j, n_1h^i+n_1)\mid i\geq 1\}\\
\cup \{(m_1h^i+m_1+h, n_1h^i+n_1)\mid i\geq 1\}
\subseteq {\mathfrak R}C_{j_0}\cap \N^2.\end{multline*}

Since ${\rm Sym}_A(F)$ is a polynomial ring over the domain $A$ and a non zerodivisor of degree ~$1$ of $A$ is a non zerodivisor of degree $(1, 0)$ in $A[Mt]$, for all $i\geq 1$ we have
$$\ell_k\big((M^{n_1h^i+n_1})_{m_1h^i +m_1}\big) \leq \ell_k\big((M^{n_1h^i+n_1})_{m_1h^i +m_1+j}\big)
\leq \ell_k\big((M^{n_1h^i+n_1})_{m_1h^i+m_1+h}\big).$$
 On the other hand,
since  $\ell_k\big((M^{n_1})_{m_1}\big) = P_{0,0}(m_1, n_1)$, we deduce
\begin{align*}
\lim_{i\to\infty}\dfrac{P_{0,0}(m_1h^i+m_1,n_1h^i+n_1)}{(n_1h^i+
n_1)^{\tilde r}} &\leq
\lim_{i\to\infty}\dfrac{P_{j,0}(m_1h^i+m_1+j,n_1h^i+n_1)}{(n_1h^i+
n_1)^{\tilde r}}\\
&\leq \lim\limits_{n\to\infty}
\dfrac{P_{0,0}(m_1h^i+m_1+h,n_1h^i+n_1)}{(n_1h^i+n_1)^{\tilde r}},
\end{align*}
which implies 
$$\sum_{l=0}^{{\tilde r}}
\lambda^0_l\left(\tfrac{m_1}{n_1}\right)^l \leq 
\sum_{l=0}^{{\tilde r}}
\lambda^j_l\left(\tfrac{m_1}{n_1}\right)^l 
\leq \sum_{l=0}^{{\tilde r}}
\lambda^0_l\left(\tfrac{m_1}{n_1}\right)^l.$$
This contradicts the choice of $(m_1, n_1)$ and thus proves Claim $(i)$.

\vspace{5pt}

 \noindent{\underline{Proof of Claim}}~(ii).\quad  Let ${\bar r} = \max\{r_{0,0}, r_{0,j}\}$, where we recall  $r_{0,j} = \deg~P_{0, j}(X, Y)$.
 Then 
$$\tfrac{P_{0,0}(X,Y)}{X^{\bar r}} =
\sum_{l=0}^{\bar r}\nu^0_l(\tfrac{Y}{X})^l + \sum_{i=1}^{\bar r}\tfrac{1}{X^i}
S^0_{{\bar r}-i}(\tfrac{Y}{X}),\quad\quad
\tfrac{P_{0,j}(X,Y)}{X^{\bar r}} =
\sum_{l=0}^{\bar r}\nu^j_l(\tfrac{Y}{X})^l + \sum_{i=1}^{\bar r}
\tfrac{1}{X^i}S^j_{{\bar r}-i}(\tfrac{Y}{X}),$$
where $S^0_i(\tfrac{Y}{X})$ and $S^j_i(\tfrac{Y}{X})$ are polynomials in $\Q[\tfrac{Y}{X}]$ of degree $\leq i$.
Note that if $r_{0,j} < {\bar r}$ (where ${\bar r}\leq r_0$) then all $\nu^j_l = 0$. Same goes if $r_{0,0} < {\bar r}$.

If the claim does not hold then $\sum_{l=0}^{\bar r}(\nu^0_l-\nu^j_l)(\tfrac{Y}{X})^l$ is a nonzero polynomial of degree $\leq {\bar r}$.
Since  $\{\frac{n+\beta_0}{(\lfloor x_0n\rfloor + \alpha_0+ih)}\mid0\leq i\leq {\bar r}\}$ is  a set of ${\bar r}+1$ distinct points there exists $0\leq i_2\leq {\bar r}$ such that for 
\begin{equation}\label{m_2n_2}
(m_2, n_2) = (\lfloor x_0n\rfloor + \alpha_0+i_2h, n+\beta_0)\quad\mbox{we have}\quad
\sum_{l=0}^{\bar r}(\nu^0_l-\nu^j_l)\left(\tfrac{n_2}{m_2}\right)^l \neq 
0.\end{equation}

Now,  by (\ref{n_x})
$$
  \{(m_2h^i, n_2h^i)\mid i\geq 1\} \cup
 \{(m_2, n_2), (m_2+jhd_1, n_2+j), (m_2+h^2d_1, n_2+h)\} 
 \subseteq \mathfrak{R}C_{j_0},$$
  and therefore 
\begin{multline}
\{(m_2h^i+m_2, n_2h^i+n_2)\} \cup 
\{(m_2h^j+m_2+jhd_1, n_2h^i+n_2+j)\}\\\
\cup\{(m_2h^j+m_2+h^2d_1, n_2h^i+n_2+h)\}
\}
\subseteq {\mathfrak R}C_{j_0}\cap \N^2.\end{multline}

By (\ref{e2_poly}), $\ell_k\big((M^{n_2})_{m_2}\big) = P_{0,0}(m_2, n_2)$ which implies
$$\ell_k\big((M^{n_2h^i+n_2+j})_{m_2h^i+m_2}\big) = P_{0,j}(m_2h^i+m_2, n_2h^i+n_2+j).$$

Let ${\tilde m}\in M$ be a nonzero homogeneous element of degree, say, $hd_1$. Then, since ${\rm Sym}_A(F)$ is a polynomial ring over an integral domain $A$ the elements    ${\tilde m}^j $ and ${\tilde m}^{h-j}$ are non zerodivisors in $A[Mt]$ of degrees $(jhd_1, j)$ and $((d-j)hd_1, d-j)$.
This induces the injective maps of $k$-vector spaces
$$(M^n)_mt^n  \longto (M^{n+j})_{m+jhd_1}t^{n+j}
 \longto (M^{n+d})_{m+h^2d_1}t^{n+d}$$
 for all $m$ and $n$.
Therefore, for $0\leq j < h$,
$$\ell_k((M^{n_2h^i+n_2})_{m_2h^i+m_2}) \leq \ell_k((M^{n_2h^i+n_2+j})_{m_2h^i+m_2+jhd_1})\leq
\ell_k((M^{n_2h^i+n_2+h})_{m_2h^i+m_2+h^2d_1}),$$
which  gives 
$$\begin{array}{lcl}
P_{0,0}(m_2h^i+m_2, n_2h^i+n_2) & \leq & P_{0,j}(m_2h^i+m_2+jhd_1, n_2h^i+n_2+j)\\\
& \leq & P_{0,0}(m_2h^i+m_2+h^2d_1, n_2h^i+n_2+h).\end{array}$$

Since $h\geq 2$, for any fixed $k$ we have
$$\lim_{i\to \infty} 
\frac{n_2h^i+n_2+k}{m_2h^i+m_2+khd_1} = \frac{n_2}{m_2}\;\;\mbox{and}\;\; 
\lim_{i\to \infty} 
\frac{m_2h^i+m_2+khd_1}{m_2h^i} = 1\
.$$
But then 
$$\sum_{l=0}^{\bar r}\nu^0_l\left(\tfrac{n_2}{m_2}\right)^l 
=
\lim_{i\to \infty} \frac{P_{0,0}(m_2h^i+m_2, n_2h^i+n_2)}{(m_2h^i)^{\bar r}}
=
\lim_{i\to \infty} \frac{P_{0,0}(m_2h^i+m_2+h^2d_1, n_2h^i+n_2+h)}{(m_2h^i)^{\bar r}}$$
and
$$ \sum_{l=0}^{\bar r}\nu^j_l\left(\tfrac{n_2}{m_2}\right)^l =  \lim_{i\to \infty} \frac{P_{0,j}(m_2h^i+m_2+jhd_1, n_2h^i+n_2+j)}{(m_2h^i)^{\bar r}} 
$$
which contradicts (\ref{m_2n_2}). This proves Claim~(ii).

\vspace{5pt} 
 
 Now to prove the last assertion of the theorem it is enough 
to show that the degree of the polynomial ${P}_{0,0}(X, Y)$ is bounded above by $d+e-2$.  We have $r_0=\deg P_{0, 0}(X,Y)$. Then

$$\tfrac{P_{0,0}(X,Y)}{Y^{r_0}} =
\sum_{l=0}^{r_0}\lambda^0_l(\tfrac{X}{Y})^{r_0-l} + \sum_{i=1}^{r_0}\tfrac{1}{Y^i}
R^0_{{r_0}-i}(\tfrac{X}{Y}),$$
where $S^0_i(\tfrac{X}{Y})$  are polynomials in $\Q[\tfrac{X}{Y}]$ of degree $\leq i$.

Since $\{\frac{\lfloor x_0n\rfloor + 
\alpha_0+ih}{n+\beta_0}\mid 0\leq i\leq r_0\}$ is a set of $r_0+1$ distinct points,
there exists $i_0\in \{0, 1, \ldots, r_0\}$ such that
$$\sum_{l=0}^{r_0}\lambda_l^0\left({m_0}/{n_0}\right)^l \neq 
0\quad\mbox{for}\quad (m_0, n_0) = (\lfloor xn\rfloor + 
\alpha_0+i_0h, n+\beta_0).$$

By (\ref{alpha}) $P_{0,0}(m_0,n_0) = 
\ell_k((M^{n_0})_{m_0})$. On the other hand
 $(m_0, n_0)\in \mathfrak{R}C_{j_0}$ implies  
$$\{(m_0(h+1)^k, n_0(h+1)^k)\mid k\geq 1\}\subset \mathfrak{R}C_{j_0}\cap \N^2.$$
So we have 
$$P_{0,0}(m_0(h+1)^k,n_0(h+1)^k) = \ell_k\big((M^{n_0(h+1)^k})_{m_0(h+1)^k}\big).$$

We reecall that the Rees algebra $A[Mt]$ independent of the choice of embedding of $M$ into a free module.
 So if we  choose a free module $E$ as given in preliminaries then  for all $m, n$ 
$$ ((sM)^n)_m\subseteq (E^n)_m \subseteq (M^n)_m\;\;\mbox{and}\;\; \ell_k((M^n)_m)\leq 
\ell_k((E^n)_{m+c_0n}),$$

where $E^n$ is a free $A$-module with basis 
$\{m_1^{i_1}\cdots m_e^{i_e}\mid \sum_ji_j = n\}$.
Since $\dim~A\geq 2$ we have
$$\ell_k((Am_1^{i_1}\cdots m_e^{i_e})_m) = 
\ell_k(A_{m-\sum_j{i_j}f_j})\leq \ell_k(A_m)
= O(m^{d-1}).$$

Therefore
\begin{equation}\label{degree}\ell_k((E^n)_{m}) \leq {{n+e-1}\choose{e-1}}\cdot \ell_k(A_{m})\implies  
\ell_k((M^n)_m)\leq {{n+e-1}\choose{e-1}}\ell_k(A_{m+c_0n}).\end{equation}

Since the Hilbert polynomial of $A$ is of degree $d-1$ if $r_0 > d+e-2$ then 
$$\lim_{k\to \infty}
\frac{\ell_k\big((M^{n_0(h+1)^k})_{m_0(h+1)^k}\big)}{(n_0(h+1)^k)^{r_0}} \leq 
\lim_{k\to \infty}\frac{{{n_0(h+1)^k+e-1}\choose {e-1}}
\ell_k(A_{(m_0+c_0n_0)(h+1)^k}\big)}{(n_0(h+1)^k)^{r_0}}
= 0.$$

This contradicts the choice of  $(m_0,n_0)$
because
$$\lim\limits_{k\to\infty}
\dfrac{P_{0,0}(m_0(h+1)^k,n_0(h+1)^k)}{(n_0(h+1)^k)^{r_0}} = 
\sum\limits_{l=0}^{r_0}\lambda_l\left(\dfrac{m_0}{n_0}
\right)^l\neq 0.$$
\end{proof}

\begin{thm}\label{t2}
Let $A$ and $M$ as  as in Notationn~\ref{n1}.  Then the function
$$f_{M}:\R\longto \R_{\geq 0}\quad\mbox{given by}\quad x\longto \lim_{n\to \infty}\frac{\ell_k\big((M^n)_{\lfloor xn\rfloor}\big)}{n^{d+e-2}/(d+e-1)!},$$
is a well defined function which is continous everywhere except possibly at  the point $x = d_1$. Moreover
$$f_{M}(x) = \begin{cases}
              0 & \text{if}\;\;x\in I(-\infty, d_{1}),\\
              {\bf p}_1(x) & \text{if}\;\;x\in I({d_1}, {d_2}],\\
              {\bf p}_j(x) & \text{if}\;\;x\in I[{d_j}, d_{j+1}]\;\;\text{where}\;\;2\leq j<l,\\
              {\bf p}_j(x) & \text{if}\;\;x\in I[{d_l},\infty),
             \end{cases}$$
where ${\bf p}_j(x)\in \Q[x]$ is a polynomial of degree $\leq d-1$ and ${\bf p}_l(x)$ is a nonzero polynomial of degree $d-1$.
\end{thm}
\begin{proof}Let $x\in I(-\infty, d_1)$, then 
 since  ${(M^{n})}_{m}=0$ for all  $m < d_1n$
we have  $f_{M}(x)=0$. 

Let $x=d_1$. Let  $L = M_{d_1}A$ denote the $A$-submodule generated by degree $d_1$ elements
of $M$. We consider $L^n$ as the image of $\tensor^nL$ in $F^n$.
Then  $(M^n)_{d_1n}
= (L^n)_{d_1n}$, where 
 $\oplus_{n\geq 0}(L^n)_{d_1n}$ is a finitely generated standard graded $k$-algebra, 
 and since  $\rank\; L\leq e$, it is of Krull dimension $\leq d+e-1$.
$ \lim_{n\to \infty}\ell_k\big((M^n)_{d_1n}\big)/n^{d+e-2}$
 exists.  In particular
 $f_M$ is well defined at $x= d_1$.

Let $x\in I({d_j}, {d_{j+1}})$ then there exists an integer $n_x>0$ such that $\left(\lfloor xn\rfloor, n\right) \in {\mathfrak R}C_j$ for all $n\geq n_x$. Therefore by Theorem \ref{t1}  there is a homogeneous polynomial $P_j(X,Y)$ of degree $r_j\leq d+e-2$ and a quasi-polynomial $Q_j(X,Y)$ of degree $<r_j$ such that for all $n\geq n_x$,
\begin{equation}\label{cont}\ell_k\big((M^n)_{\lfloor xn\rfloor}\big) = P_j({\lfloor xn\rfloor}, n) + Q_j({\lfloor xn\rfloor},n).\end{equation}
This implies that 
\begin{equation*}f_{M}(x) = \lim_{n\to \infty}\frac{\ell_k\big((M^n)_{\lfloor xn\rfloor}\big)}{n^{d+e-2}/(d+e-1)!} = \lim_{n\to\infty}\dfrac{P_j({\lfloor xn\rfloor}, n)}{n^{d+e-2}/(d+e-1)!} =:{\bf p}_j(x)\;\;\mbox{if}\;\;x\in I(d_j, d_{j+1})\end{equation*}
 
 Hence $f_M(x)$ 
is a well defined function continous function on  the open intervals $I(d_j, d_{j+1})$ for all  $1\leq j\leq l$, where we define $d_{l+1} = \infty$.
So to prove that the function
$f_M$ is well defined and continous on the interval $I(d_1, \infty)$   it is enough to show that
 $$\lim_{n\to \infty}\frac{\ell_k\big((M^n)_{d_i})}{n^{d+e-2}/(d+e-1)!} = {\bf p}_{i-1}(d_i) = {\bf p}_{i}(d_i),\;\; \mbox{for}\;\; 2\leq i\leq l.$$

Let $x=d_i$, where $2\leq i \leq l$. 
By Lemma~\ref{cgl2}  there exist integers ${\tilde m_0}, m_1\in \N_{>0}$ such that for all $n\gg 0$,
$$\begin{array}{lcl}
 (xn+{\tilde m_0}d_1, n+{\tilde m_0}) & = &  n(x, 1)+ {\tilde m_0}(d_1, 1) \in {\mathfrak{R}}C_{i-1}\\
 (xn-m_1, n) & = & n(x, 1) -  m_1(1, 0) \in {\mathfrak{R}}C_{i-1}.
\end{array}$$

So by (\ref{cont}) 
$$\lim_{n\to \infty}\frac{\ell_k\big((M^n)_{xn-m_1}\big)}{n^{d+e-2}/ (d+e-1)! } = {\bf p}_{i-1}(x) = \lim_{n\to \infty}\frac{\ell_k\big((M^{n+{\tilde m_0}})_{ xn+{\tilde m_0}d_1}\big)}{n^{d+e-2}/(d+e-1)!  }. 
$$

Let $z_1\in M$ be a nonzero element of degree $(d_1, 1)$ then $z_1^{\tilde m_0}\in M^{\tilde m_0}\subseteq F^{\tilde m_0}$ is a non zerodivisor of degree $({\tilde m_0}d_1, {\tilde m_0})$
on $A[Mt]$ . Also by choosing a nonzero element $z_0$ of degree $1$ in $A$ we get a nonzero divisor $z_0^{m_1}$ of degree $(m_1, 0)$ on $A[Mt]$.
This gives injective maps
\begin{equation}\label{conte1}
(M^n)_{xn-m_1}\longby{z_0^{m_1}}
(M^n)_{xn}\longby{z_1^{\widetilde m_0}}
(M^{n+{\tilde m_0}})_{xn+{\tilde m_0}d_1}.
\end{equation}
Now the equality 
\begin{multline*}
  \lim_{n\to \infty}\frac{\ell_k\big((M^n)_{xn-m_1}\big)}{n^{d+e-2}} \leq \liminf_{n\to \infty}\frac{\ell_k\big((M^n)_{xn}\big)}{n^{d+e-2}}  \leq \limsup_{n\to \infty}\frac{\ell_k\big((M^n)_{xn}\big)}{n^{d+e-2}} \leq \lim_{n\to \infty}\frac{\ell_k\big((M^{n+{\tilde m_0}})_{ xn+{\tilde m_0}d_1}\big)}{n^{d+e-2}}
\end{multline*}
implies that 
$$ f_{M}(x) = \lim_{n\to \infty}\frac{\ell_k\big((M^n)_{xn}\big)}{n^{d+e-2}/(d+e-1)!  } = {\bf p}_{i-1}(x).$$

Similarly, by Lemma~\ref{cgl3} there exist integers $m_0 = {\tilde m_0}$ and $ m_1\in \N_{>0}$ such that for all $n\gg 0$,
\begin{align*}
  n(x, 1) - {\tilde m_0}(d_1, 1) \in {\mathfrak{R}}C_{i},\;\;\mbox{and}\;\;
  n(x, 1) + m_1(1, 0) \in {\mathfrak{R}}C_{i}.
\end{align*}
Again choosing non zero divisors $z_1$ and $z_0$ of degrees $(d_1, 1)$ and $(1, 0)$ on $A[Mt]$ we get injective maps
\begin{equation}\label{conte2}
(M^{n-{\tilde m_0}})_{xn-{\tilde m_0}d_1}\longby{z_1^{\widetilde m_0}}
(M^n)_{xn}\longby{z_0^{m_1}}
(M^{n})_{xn+m_1}.
\end{equation}
Again by (\ref{cont}) 
$$\lim_{n\to \infty}\frac{\ell_k\big((M^n)_{xn+m_1}\big)}{n^{d+e-2}/(d+e-1)! } = {\bf p}_i(x) = \lim_{n\to \infty}\frac{\ell_k\big((M^{n-{\tilde m_0}})_{xn - {\tilde m_0}d_1}\big)}{n^{d+e-2}/ (d+e-1)! }.$$
Therefore,
$$ f_{M}(x) = \lim_{n\to \infty}\frac{\ell_k\big((M^n)_{xn}\big)}{n^{d+e-2}/ (d+e-1)! } = {\bf p}_i(x).$$
This proves the continuity of $f_{M}$ at $x=d_i$ and therefore on the interval $(d_1, \infty)$

To prove the assertion of the degrees of ${\bf p}_j(x)$, we choose $E$ as in Section~1.
By (\ref{degree})
$$\ell_k((M^n)_m) \leq {{n+e-1}\choose{e-1}}\ell_k(A_{m+c_0n})$$
which implies 
that  each ${\bf p}_j(x)$ is a polynomial of degree $\leq d-1$.

Now we consider the polynomial ${\bf p}_l(x)$, where by our convention  $d_l = d_M$.

We recall that for  $m\gg 0$, we have  $\ell_k(A_m) = {\tilde e_0}m^{d-1}+{\tilde e_1}m^{d-2}+\cdots +{\tilde e_{d-1}}$, 
the  Hilbert polynomial of $A$.
On the other hand for all $x>d_M$
$${\bf p}_l(x) = \lim_{n\to \infty}\frac{\ell_k((M^n)_{\lfloor xn\rfloor}}{n^{d+e-2}} \geq 
\lim_{n\to \infty}\frac{\ell_k((E^n)_{\lfloor xn\rfloor}}{n^{d+e-2}} \geq \frac{
{\tilde e_0}(x-d_M)^{d-1}+{\tilde e_1}(x-d_M)^{d-2}+\cdots +{\tilde e_{d-1}}}{(e-1)!}.$$
Therefore  the polynomial ${\bf p}_l(x)$ is a polynomial of degree $\geq d-1$ and hence of degree $= d-1$.

\end{proof}

  \begin{defn}\label{adic-density}The 
{\em adic density function} of $M$ is the function
$$f_{M}:\R\longto \R_{\geq 0}\quad\mbox{given by}\quad 
x  \to \lim_{n\to \infty}f_n(x),$$
where $f_n:\R\longto \R_{\geq 0}$ is given by
$$f_n(x) = \frac{\ell_k(({M}^n)_{\lfloor xn\rfloor})}{n^{d+e-2}/(d+e-1)! }.$$
We proved in Theorem~\ref{t2} that it is a piecewise polynomial function 
which is supported on $I[d_1, \infty)$ and is continuous on $I(d_1, \infty)$.
 \end{defn}
 
\vspace{5pt}

\section{Saturated density function} 

We follow Notations~\ref{n1}, where 
 $M$ is a torsion free $\N$-graded module with graded injective map $M\longrightarrow F$ 
 of $A$-modules where $F$ is  a graded free $A$-module. However it may happen that $F$ has a nonzero negative degree components. Therefore even though $M$ is $\N$-graded its saturation 
$(M:_{F}{\bf m}^{\infty})$ may have negative degree components. As a result unlike in the case of ideals in \cite{DRT25}, here the saturation density function for modules may no longer be supported on $\R_{\geq 0}$.

 \vspace{5pt}
 
 \noindent{\bf  Example}.\quad Let $I = (x^2, xy)$ be an ideal in the polynomial ring $A = k[x, y]$. Then $M= I(2)$ is a $\N$-graded module with graded embedding $M = I(2)\longto F = A(2)$, and $F_{-2}\neq 0$ and $F_{-1}\neq 0$. 
 On the othe hand $x \in ((x^2, xy):_{A(2)}
 {\bf m}^{\infty}) = 
 (M:_{F}{\bf m}^{\infty})$. In particular $
 (M:_{F}{\bf m}^{\infty})_{-1}\neq 0$.

We recall the following result from \cite[Corollary~2.3]{Cut11}.
\begin{cor}
 Suppose $W$ is a Noetherian scheme and $\sB = \oplus_{k\geq 0}\sB_k$ a finitely generated graded $\sO_W$-algebra, which is locally generated by $\sB_1$
 as 
 a $\sO_W$-algebra. Let $W' = \mbox{\bf Proj}_W\sB$
and $\alpha:W'\longto W$ the structure morphism.
 Then there exists a positive integer ${\bar k}$ such that $\sB_k = \alpha_*\sO_{W'}(k)$ for all $k\geq {\bar k}$.
 \end{cor}

The next theorem is a fundamental result which shows that the $\limsup$ in the definition of volume can be replaced by a limit. Under the assumptions that $k$ is algebraically clsoed, it has been proven by Okounkov \cite{Oko03} for an ample divisor $D$ and later by Lazarsfeld-Musta\c{t}\u{a} \cite{LM09} when $D$ is a big divisor. In the following generality, it was proved by Cutkosky \cite{Cut14}.

\begin{thm}\label{volume} Let $X$ be a $(d-1)$-dimensional projective variety over a field $k$ and $D\in \Div(X)$ be an integral divisor. Then
 $$\mathrm{vol}_X(D) = \limsup_{n\to \infty} \frac{h^0(X, \sO_X(nD))}{n^{d-1}/(d-1)!} = \lim_{n\to \infty} \frac{h^0(X, \sO_X(nD))}{n^{d-1}/(d-1)!}.$$
\end{thm}

 The continuity property  of the volume function (see \cite[Theorem~2.2.44]{Laz04b}  implies that the above notion  extends uniquely to a continuous function
 ${\rm vol}_X:N^1(X)_{\R}\longto \R_{\geq 0}$.

\begin{thm}\label{t3}
The function
${\tilde f}_{M}:\R\longto \R_{\geq 0}$ given by
$$
x\to \limsup_{n\to \infty}\frac{\ell_k(M^n:_{F^n}{\bf m}^{\infty})_{\lfloor xn\rfloor}}{n^{d+e-2}/
 (d+e-1)! }$$
is a continuous function and supported on $I(-c_0, \infty)$, for some  constant $c_0\geq 0$.
Further there is a polynomial ${\tilde P_M}(x)$ of degree $d-1$ such that 
${\tilde f}_M(x) = {\tilde P_M}(x)$ for all $x\geq  d_M$.
\end{thm}

\begin{proof}We consider the embeddding $M\longto F$ of graded modules, where $F$ is a free $A$-module of rank $e$. We write 
$$M\subseteq F = A(-c_0)\oplus A(-c_1)\oplus \cdots 
\oplus A(-c_{e-1}),$$
where $-c_0\leq \cdots \leq -c_{e-1}$.
In particular 
$F$ has no nonzero component of degree $<-c_0$.

Let $V = {\rm Proj}~A$ and let $\sM$ denote the sheaf of $\sO_V$-modules associated to the graded modules $M$.
 Since $A$ is an integral domain we know $H^0_{\bf m}(A) = 0$ and there exists sequence of graded maps of $k$-vector spaces
\begin{equation*}0\longto \oplus(H^0_{\bf m}(A))_m \longto 
\oplus_{m\geq 0}A_m\longto \oplus_{m\in \Z}H^0(V, \sO(m))\longto \oplus_{m\in \Z}(H^1_{\bf m}(A))_m \longto 0.\end{equation*}
 So 
 $\ell_k(\oplus_{m\in \Z}(H^1_{\bf m}(A))_m) < \infty$ and in particular there exists a constant $h_1$ such that for all $m\in \Z$  and $n\geq 0$ we have 
 $$\ell_k(H^0_{\bf m}(F^n))_m = 0, \quad 
 \ell_k(H^1_{\bf m}(F^n))_m \leq  h_1n^{e-1}.$$

Now the short exact sequence of $A$-modules
 $$0\longto M^n\longto F^n\longto F^n/M^n\longto 0$$
 gives  the graded exact sequence of $k$-vector spaces
 \begin{equation}\label{lge0}
 0\longto \oplus_{m\in \Z} (H^0_{\bf m}(F^n/M^n))_m \longto 
\oplus_{m\in \Z} (H^1_{\bf m}(M^n))_m\longto 
\oplus_{m\in \Z} (H^1_{\bf m}(F^n))_m,
\end{equation}
where by definition 
$ (H^0_{\bf m}(F^n/M^n))_m =
((M^n:_{F^n}{\bf m}^{\infty})_m/M^n)_m$,
so we have 
\begin{equation}\label{lge1}
\ell_k((H^1_{\bf m}(M^n))_m) 
 =
\ell_k({(M^n:_{F^n}{\bf m}^{\infty})_m})-\ell_k({(M^n)_m}) +O(n^{e-1}).\end{equation}

On the other hand, for a fixed $n\geq 0$, 
we have the graded exact sequence of $A$-modules
\begin{equation}\label{lge2}0\longto H^0_{\bf m}(M^n) \longto 
M^n\longto \oplus_{m\in \Z}H^0(V, \sM^n(m))\longto H^1_{\bf m}(M^n) \longto 0.\end{equation}
So
 $\ell_k((H^1_{\bf m}(M^n))_m) = \ell_k(H^0(V, \sM^n(m)))-\ell_k((M^n)_m)
 $, using (\ref{lge1}) we get  
\begin{equation}\label{saturation}
0\leq \ell_k(H^0(V, \sM^n(m))
-\ell_k((M^n:_{F^n}{\bf m}^{\infty})_m)
\leq  O(n^{e-1})\quad\mbox{for all}\quad m, n.
\end{equation}

Now we can interpret 
$\ell_k(H^0(V, \sM^n(m))$ in terms of volume function as follows.
Let $\sB = \oplus_{n\geq 0}\sM^nt^n$ denote the standard graded $\sO_V$-algebra generated by $\sM t$.
This gives  a proper  map  
$\pi:X= {\bf Proj}_V{\sB} \longto V$.  
Let $H$ denote  the pull back of a hyperplane section on $V$ then 
$\sO_X(H) = \pi^*\sO_V(1)$. Further there is a canonical  surjective map $\pi^*(\sM) \longto 
\sO_X(1)  = \sO_X(E)$ of $\sO_X$-modules, where 
$\sO(E)$ is a Cartier but (unlike in the ideal case)  not necessarily an effective 
divisor  on $X$. However,
(see \cite[Corollary~2.3]{Cut11}), still we have ${\bar n}$ such that  for $n\geq {\bar n}$ we have $\sM^n = \pi_*\sO_X(nE)$.
Combining this with the projection formula we get that for all $n\geq {\bar n}$
$$H^0(X, \sO_X(mH+nE)) = 
H^0(X, \pi^*\sO_V(m)\tensor_{\sO_X}\sO_X(nE)) 
= H^0(V, \sM^n(m)).$$

Since $X$ is a $d+e-2$ dimensional projective variety, where $d\geq 2$   the volume function 
is a well defined continuous function on $N^1(X)_{\bf \R}$, such that  volume function on the set $\{xH+E\mid x\in \R\}$ is given by 
\begin{equation}\label{volume}{\rm vol}_X(xH+E) = \lim_{n\to \infty}
\frac{h^0(X, \sO(\lfloor xn\rfloor H+nE))}{n^{d+e-2}/(d+e-2)!}\quad\mbox{for all}\;\; x\in \R.\end{equation}

Therefore, for all $x\in \R$, by (\ref{saturation}) 
$${\tilde f}_M(x)  =
\limsup_{n\to \infty}\frac{\ell_k(M^n:_{F^n}{\bf m}^{\infty})_{\lfloor xn\rfloor}}{n^{d+e-2}/(d+e-1)!  } = (d+e-1)
{\rm vol}_X(xH+E)
$$
and hence $f_M$ is a continuous function on $\R$. 
Since $F^n$ has no nonzero component of degree $<-nc_0$, the function ${\tilde f_M}$ is supported on the interval $I[-c_0, \infty)$.

Now let $d_1< d_2< \cdots < d_l=d_M$ denote the degrees of a set of homogeneous generators of $M$ and let $a_i$ denote the number of degree $d_i$ generators.
Then we have a canonical surjective map of $\sO_V$-modules
$$\sO_V(-d_1)^{\oplus a_1}\oplus
\sO_V(-d_2)^{\oplus a_2}\oplus
\cdots \oplus 
\sO_V(-d_M)^{\oplus a_l}\longto \sM$$
which gives 
surjective maps of $\sO_X$-modules
$$\pi^*\big(\sO_V(d_M-d_1)^{\oplus a_1}\oplus
\sO_V(d_M-d_2)^{\oplus a_2}\oplus
\cdots \oplus 
\sO_V^{\oplus a_l}\big)\longto \pi^*(\sM(d_M))\longto  \sO_X(d_MH+E).$$

In particular  $\sO_X(d_MH+E)$ is 
globally generated and hence $d_MH+E$ is a  nef divisor. Now 
by the
asymptotic Riemann-Roch, if $D$ is a nef divisor in $N^1(X)_{\R}$ then  
$$h^0(X, \sO_X(nD)) = \frac{(D^{d+e-2})}{(d+e-2)!}\cdot n^{d+e-2}+ O(n^{d+e-3}).$$

Then for $x>d_M$ 
\begin{equation}\label{satpoly}{\tilde f}_M(x) = \sum_{i=0}^{d+e-2}(-1)^i\frac{(d+e-1)!}{(d+e-2-i)!i!}(H^{d+e-2-i}\cdot E^i)x^{d+e-2-i}.\end{equation}

Since $H^{j}=0$ for $j>d-1$,  we can write it as  
\begin{equation}\label{satpoly1}{\tilde f}_M(x) = (d+e-1)!\left[\frac{(H^{d-1}\cdot (E)^{e-1})}{(d-1)!(e-1)!}x^{d-1}
+\frac{(H^{d-2}\cdot (E)^{e})}{(d-2)!(e)!}x^{d-2}+\cdots +\frac{((E)^{d+e-2})}{(d+e-2)!}\right]
.\end{equation}

On the other hand, for $x\gg 0$ the function $f_M(x)$ is also a polynomial of degree $d-1$, so the inequality  $f_M\leq {\tilde f_M}$ implies that  the degree of the above polynomial is exactly $d-1$.
\end{proof}

\vspace{5pt}

\begin{defn}\label{saturated-density}
 The {\em saturated density function} of $M$ is the function  
$${\tilde f}_M:\R\longto \R_{\geq 0} \quad\mbox{given by}\quad 
x \to  \lim_{n\to \infty}g_n(x),$$
 where 
  ${g}_n:\R\longto \R_{\geq 0}$ is the function given by
$${g}_n(x) = \frac{\ell_k((\tilde {M^n})_{\lfloor xn\rfloor})}{n^{d+e-2}/(d+e-1)! },\;\;\mbox{where}\;\;
{\tilde {M^n}} = (M^n:_{F^n}{\bf m}^{\infty}).$$
We know by Theorem~\ref{t3} that it is a continous function and is 
 supported on $I[-c_0, \infty)$, for some constant $c_0\geq 0$.
\end{defn}

 \section{$\varepsilon$-density functiom}
 
We recall that in our context Notationa~\ref{n1}, 
for any choice of embedding of $M$ in a graded free module $F$ the $A$-module $M^n = {\rm Sym}^n_AN/{\rm Tor}_A({\rm Sym}^n_AN)$.  

We recall a fundamental result due to V. Kodiyalam.
\begin{thm}\label{kodi}\cite[Corollary~2]{Kodi00} Let $N$ be a finitely generated graded module over a standard graded  ring $A$.
Let $S_n(N) = {\rm Sym}^n_AN/{\rm Tor}_A({\rm Sym}^n_AN)$, where 
${\rm Sym}^n_AN$ denotes the $n^{th}$ symmetric power of $N$ as an $A$-module.
Let 
$${\rm reg}(N) = \max\{a_i+i\mid a_i\;\; \mbox{is the largest integer such that}\;\;H^i_{\bf m}(N)_{a_i}\neq 0\}.$$
Then
there is an integer $\tau$ such that
${\rm reg}(S_n(N))\leq \tau n$ for all $n\in \N$.
\end{thm}

\noindent{\bf Lebesgue's dominated convergence theorem}. {\em Let  $\{f_n:\R\longto \R\}_n$ be a sequence of measurable functions. Suppose that the sequence converges pointwise to a function $f:\R\longto\R$ and  the sequence $\{f_n\}_n$ is dominated by an integrable function 
 $g:\R\longto \R$, {\em i.e.},
$|f_n(x)|\leq g(x)$
for all points $x\in \R$.
 Then $f$ and $\{f_n\}_n$ are integrable and 
$$\lim_{n\to \infty} \int_{\R}f_n d\mu = \int_{\R} \lim_{n \to \infty} f_n d\mu= 
\int_{\R}  f d\mu $$
}

\begin{thm}\label{l1} 
 The function $f_{\varepsilon(M)}:\R\longto \R_{\geq 0}$ given by
 $$x\to \lim_{n\to \infty} \frac{\ell_k(({\tilde M^n}/M^n)_{\lfloor xn\rfloor})}{n^{d+e-1}/ (d+e-1)! }$$
is a well defined compactly supported function with support in $I[-c_0, d_M]$, where $c_0\geq 0$ is given as in Theorem~\ref{t3} such that 
\begin{enumerate}
                                                                            \item the function $f_{\varepsilon(M)}$                                                                    is continous everywhere possibly except at $x=d_1$, and  
\item 
$$\varepsilon(M) = 
\int_{-\infty}^{\infty}f_{\varepsilon(M)}(x)dx = 
\int_{-c_0}^{d_M}f_{\varepsilon(M)}(x)dx = 
 \int_{-c_0}^0{\tilde f}_M(x)dx +
 \int_{0}^{d_M}f_{\varepsilon(M)}(x)dx.$$
\end{enumerate}
\end{thm}

\begin{proof} Assertion $(1)$  follows from the properties
of the adic density function $f_M$ and 
saturated density function ${\tilde f_M}$.

 Note that $M^n = S_n(M)$, where $S_n(M)$ is as in Theorem~\ref{kodi}. Therefore there exists $\tau$  such that $H^1_{\bf m}(M^n)_m = 0$ for all $m\geq \tau n$. Note that $H^0_{\bf m}(M^n)_m = 0$, for all $m$.
Therefore 
\begin{equation}\label{*1}(M^n:_{F^n}{\bf m}^{\infty})_m  = (M^n)_m, \quad\mbox{for all}\quad  m\geq \tau n.\end{equation}
Also, if $c_0\geq 0$ such that $F_m= 0$ for all $m < -c_0$ then 
\begin{equation}\label{*2} g_n(x) = \ell_k(({\tilde M^n})_{\lfloor xn\rfloor}) = 0\quad\mbox{for}\quad x <-c_0.\end{equation}
Therefore $f_n(\varepsilon)(x) = g_n(x) -f_n(x) = 0 $ for all $x$ outside the interval $I[-c_0, \tau]$.

On the other hand by 
Theorem~\ref{t2} and Theorem~\ref{t3} 
if $x\geq d_M$ then
$$f_M(x) = P_M(x)\quad\mbox{and}\quad {\tilde f}_M(x)
= {\tilde P}_M(x),$$ where 
 $P_M(x)$ and ${\tilde P}_M(x)$ are polynomials. Therefore ${\tilde P}_M(x)
 =P_M(x)$, and the $\varepsilon$-density function 
$$f_{\varepsilon}(x) = 
 f_{\tilde M}(x) - f_{M}(x) = 0\quad\mbox{for}\quad  x>d_M.$$ 

 On the other hand, by (\ref{*1}) and (\ref{*2})
$$\varepsilon(M)   = 
\displaystyle{ \lim_{n\to \infty}\sum_{m=-c_0n}^
 {\tau n} \frac{\ell_k(({\tilde M^n}/M^n)_{\lfloor xn\rfloor})}{n^{d+e-1}/ (d+e-1)! }} = 
 \lim_{n\to \infty}\int_{-c_0}^{\tau}(g_n(x)-f_n(x))dx= 
\lim_{n\to \infty}\int_{-c_0}^{\tau}f_n(\varepsilon)(x)dx.$$ 
 
 Now we would like to apply  the Lebesgue dominated convergence theorem.
We note that 
by (\ref{degree}), if $x\in I(0, \tau)$ then   for all $n$ there is $C$ (independent of $x$) such that 
$f_n(x) = |f_n(x)|\leq C$.

Further for any $m/n\in I[-c_0, \tau)$
by (\ref{saturation})
we have
$$ (M^n:_{F^n}{\bf m}^{\infty})_{m}
\leq  h^0(X, \sO_X(mH+nE)) \leq h^0(X, n(\tau H+E)).$$
Wheras by the  asymptotic Riemann-Roch theorem applied to the nef divisor $\tau H+E$, we get 
 a constant $\lambda$ such that $|g_n(x)|\leq \lambda$ 
 for $x\in I[-c_0, \tau]$.

 Therefore for the sequence
 $\{f_n(\varepsilon) = g_n-f_n:\R\longto \R\}_n$ which is supported on the interval $(-c_0, \tau)$ we get 
$$\varepsilon(M) = 
\lim_{n\to \infty}\int_{-c_0}^{\tau}f_n(\varepsilon)(x)dx =
\int_{-c_0}^{\tau}\lim_{n\to \infty}f_n(\varepsilon)(x) 
=
\int_{-c_0}^{\tau}
 f_{\varepsilon(M)}(x)dx = 
\int_{-c_0}^{d_M}
 f_{\varepsilon(M)}(x)dx.$$
\end{proof}

\begin{defn}\label{epsilon-density}
 The {\em epsilon density ($\varepsilon$-density) function} of $M$ is the function $$f_{{\varepsilon}(M)}:\R\longto \R_{\geq 0} \quad\mbox{given by}\quad 
x \to \lim_{n\to \infty}f_n(\varepsilon)(x),$$
where $f_n(\varepsilon):\R \longto \R_{\geq 0}$ is the function 
given by
$$ f_n(\varepsilon)(x) = g_n(x)-f_n(x) = \frac{\ell_k(({\tilde M^n}/M^n)_{\lfloor xn\rfloor})}{n^{d+e-2}/(d+e-1)!}. $$

In particular
$f_{{\varepsilon}(M)}(x)= {\tilde f_M}(x)-f_M(x)$ 
 is continuous on $\R\setminus \{d_1\}$ and is supported on $I[-c_0, d_M]$, for some constant 
$c_0\geq 0$.
\end{defn}

\begin{rmk}\label{r2} In the above theorem we have seen that the functions
 $f_M$ and ${\tilde f_M}$ are the same polynomial function on the interval $(d_M, \infty)$ and given by
 the polynomial as in (\ref{satpoly1}).
 We denote this polynomial as
 $$P_M(x) = \frac{e_{d-1}(A(Mt])}{(d-1)!(e-1)!}x^{d-1}+
 \frac{e_{d-2}(A(Mt])}{(d-2)!(e)!}x^{d-2}+
 \cdots+\frac{e_0(A[Mt])}{(d+e-2)!},$$
 and where  $e_i(A[Mt])$ is known (in \cite{HT03}) as the $i^{th}$ mixed multiplicity of $A[Mt]$. Therefore 
 $$e_{d-i}(A[Mt])
 = (d+e-1)!(-1)^{e+i-2}\left(H^{d-i}\cdot E^{e+i-2}\right).$$

 In particular, if we denote 
 $$P_M(X, Y) = Y^{e-1}\left[\frac{e_{d-1}(A(Mt])}{(d-1)!(e-1)!}X^{d-1}+
 \frac{e_{d-2}(A(Mt])}{(d-2)!(e)!}X^{d-2}Y+
 \cdots+\frac{e_0(A[Mt])}{(d+e-2)!}Y^{d-1}\right],$$
 then  by
 Corollary~\ref{HT},  there is  $n_0\geq 0$ such that 
 $$\ell_k((M^n)_m) = P_M(m,n)+Q_M(m,n)\quad\mbox{
whenever}\quad  m\geq n(d_M)+n_0,$$

where $Q_M(X,Y)$ is a polynomial in $\Q[X, Y]$ of degree $<d+e-2$.
 \end{rmk}

\section{Numerical characterization of integral  dependence of modules}

\begin{thm}\label{t4}Let $N\subseteq M$ be a $\N$ graded modules with notations as in \ref{d1} and \ref{n1}.

If $N$ is a reduction of $M$, that is 
 $M^{n_0+1} = NM^{n_0}$, for some $n_0$, then 
 \begin{enumerate}
  \item $\rank~N = \rank~M = e$.
\item $f_N \equiv f_M$ and 
\item ${\tilde f}_N \equiv {\tilde f}_M$.\end{enumerate}
\end{thm}

\begin{proof}By hypothesis $M^{n+1} = NM^n$ for all $n\geq n_0$ and $M^{n+n_0} = N^nM^{n_0}$ for all $n\geq 0$. If $\rank~N< \rank~M$ then $\rank~NM^n< \rank~M^{n+1}$ for all $n\geq 1$. Therefore 
 $\rank~N = \rank~M = e$ and hence the Assertion~(1).

Now since $\rank~N = e$, the adic density function 
$f_N:\R\longto \R_{\geq 0}$ and the saturated density function ${\tilde f_N}:\R\longto \R_{\geq 0}$ is given as
$$f_N(x) = \lim_{n\to \infty}\frac{\ell_k(({N}^n)_{\lfloor xn\rfloor})}{n^{d+e-2}/(d+e-1)!  }\quad\mbox{and}\quad
{\tilde f_N}(x) = \lim_{n\to \infty}\frac{\ell_k(({\tilde N^n})_{\lfloor xn\rfloor})}{n^{d+e-2}/ (d+e-1)! }
.$$
 
 Moreover,
if $d_1$ is the least degree of a generator of $M$ then so is for $N$ because we have 
  $$0\neq (M^{n+1})_{d_1(n+1)} = 
  N_{d_1}(M^n)_{d_1n}\quad\mbox{which implies}\quad N_{d_1} \neq 0.$$
  Since  $N^{n_0}\subseteq M^{n_0}$ have the same rank,
  there is a  homogeneous element $r_0\in A\setminus \{0\}$ of degree, say $m_0$, such that $r_0M^{n_0}\subseteq N^{n_0}$ and therefore  $r_0M^n\subseteq N^n$ for all $n\geq n_0$.

 \vspace{5pt}
 
 \noindent{\bf Assertion}~(2).\quad 
 Assume $x\in (d_1, \infty)$, more specifically $x\in 
 I(d_j, d_{j+1}]$ for some $j\geq 1$, where we recall that $d_1< d_2<\cdots < d_l = d_M$ are the degrees of a minimal set of homogeneous generators for $M$ and $d_{l+1} = \infty$.
 For  $n\gg 0$ we have 
 $d_jn < \lfloor xn \rfloor -m_0\leq d_{j+1}$.
 Since $r_0$ is a non zerodivisor 
of degree $(m_0, 0)$ in ${\rm Sym}_A(F)$ 
and $(r_0M^n)_{\lfloor xn\rfloor}=
r_0(M^n)_{\lfloor xn\rfloor-m_0}
\subseteq (N^n)_{\lfloor xn\rfloor}$ for all $n> n_0$, we have
 $$\ell_k(({M}^n)_{\lfloor xn\rfloor-m_0})
 \leq \ell_k(({N}^n)_{\lfloor xn\rfloor})\leq 
 \ell_k(({M}^n)_{\lfloor xn\rfloor}).$$
 By Theorem~\ref{t3}
 $$f_M(x) = {\bf p}_j(x) \leq f_N(x)\leq {\bf p}_j(x) = f_M(x).$$
 This implies $f_M(x) = f_N(x)$ for $x\in I(d_1, \infty)$.
 Now if  $x=d_1$ then for all $n\geq 1$, we have
 $$(M_{d_1})^n =
 (M^n)_{d_1n} = (N^{n-n_0}M^{n_0})_{d_1n} = (N^{n-n_0})_{d_1(n-n_0)}(M^{n_0})_{d_1n_0} = (N_{d_1})^{n-n_0}(M_{d_1})^{n_0}.$$

Now 
$\oplus_{n\geq 0} (N_{d_1})^n \longto 
\oplus_{n\geq 0} (M_{d_1})^n$ is an integral extension of standard graded $k$-algebras they have same multiplicity, that means
$$f_M(d_1) =  
\lim_{n\to \infty}\frac{\ell_k((M^n)_{d_1n}}{n^{d+e-2}/(d+e-1)!}
= \lim_{n\to \infty}
\frac{\ell_k((N^n)_{d_1n)}}{n^{d+e-2}/(d+e-1)!} = f_N(d_1).$$
This gives the equality $f_M\equiv f_N$ as 
$f_M(x) = f_N(x) = 0$
for $x\in (0, d_1)$.

 \vspace{5pt}
 
 \noindent{\bf Assertion}~(3).\quad  
 Now for all $n > n_0$ the inclusion $r_0M^{n} \subseteq N^n$ implies that $r_0({\tilde {M^n}})\subseteq
{\tilde {N^n}}\subseteq {\tilde {M^n}}$.
Now for any $x\in \R$
 $$\lim_{n\to \infty}\frac{\ell_k(({\tilde M^n})_{\lfloor xn\rfloor-m_0})}{n^{d+e-2}/ (d+e-1)! }
 \leq \lim_{n\to \infty} \frac{\ell_k(({\tilde N^n})_{\lfloor xn\rfloor})}{n^{d+e-2}/(d+e-1)!  }\leq 
\lim_{n\to \infty} \frac{\ell_k(({\tilde M^n})_{\lfloor xn\rfloor})}{n^{d+e-2}/(d+e-1)!  }.$$

But both extreme limits are equal due to the continuity property of ${\tilde f}_M$ which gives ${\tilde f_M} \equiv {\tilde f_N}$.
\end{proof}

\vspace{5pt}

\begin{rmk}\label{generators}
 Let $c> d_M$ be an integer then 
 $(M_c)^n = (M^n)_{cn}$. 
 This is because if  $m_1 , \ldots, m_s$ are homogeneous generators of $M$ of  degrees $d_1\leq \cdots \leq d_s = d_M$, respectively, then
 $$cn -\sum_jd_ji_j = \sum_j((c-d_j)i_j)\quad\mbox{and}\quad 
 A_{cn-\sum_jd_ji_j} = \prod_{j}(A_{c-d_j})^{i_j}.$$
 Therefore
 $$(M^n)_{cn} = 
 \sum_{i_1+\cdots +i_s=n}(A_1)^{cn-\sum_jd_ji_j}m_1^{i_1}\cdots m_s^{i_s}
 =  \sum_{i_1+\cdots +i_s=n}(A_{c-d_1}m_1)^{i_1}\cdots (A_{c-d_s}m_s)^{i_s}
 = (M_c)^n.$$
 In particular 
 $(M_cA)^n = (M_c)^nA = (M^n)_{\geq cn}$.
\end{rmk}

We recall the following result of \cite[Theorem~3.3]{SUV01}, the result stated here is taken from \cite{HS06}].

\begin{cor}\label{SUV}
Let $A = \oplus_{i\geq 0}A_i\subseteq B = \oplus_{i\geq 0}B_i$ be a homogeneous
inclusion of graded Noetherian rings such that $R = A_0 = B_0$ is a local ring
with maximal ideal ${\bf m}$ and such that $A = R[A_1]$ and $B = R[B_1]$. 
 Set $d = \dim~B$.
Assume further that $\lambda_R(B_1/A_1)< \infty$.
Then for all $n\gg 0$, $\lambda_R(B_n/A_nB_0)$ is a polynomial function
$p_1(n)$ of degree
$d-1$ and 
the polynomial $p_1(n)$ has the form
$$p_1(n) = \frac{e_1(A, B)}{(d-1)!} n^{d-1} + O(n^{d-2}).$$

Further if $B$ is integral over $A$, and $B_{\bf q} = A_{\bf q}$
for every minimal prime ideal ${\bf q}$ in $A$, then $e_1(A, B) = 0$. Conversely, if $B$ is
equidimensional, universally catenary, and $e_1(A, B) = 0$, then $B$ is integral
over $A$ and $B_{\bf q} = A_{\bf q}$ for every minimal prime ideal ${\bf q}$ in $A$.\end{cor}

\begin{propose}\label{p1}
Let $N\subseteq M$ be a $\N$-graded $A$-modules as in \ref{d1} and \ref{n1}.
If $\rank\;M = \rank\;N$  and ${\tilde f_M}(c) = {\tilde f_N}(c)$ for some $c > \max\{d_M, d_N\}$ then 
\begin{enumerate}
 \item 
${\tilde f_M} \equiv {\tilde f_N}$.
\item The canonical map of $(c,1)$-diagonal subalgebras
$$A[{N}t]_{\Delta_{(c,1)}} = \oOplus_{n\geq 0}(N_c)^{n}t^{n}\longto A[{M}t]_{\Delta_{(c,1)}} = \oOplus_{n\geq 0}(M_c)^{n}t^{n}$$
is an integral extension of rings.
\end{enumerate}

\end{propose}
\begin{proof}We note $(N^n)_{cn} = (N_c)^n$ and $(M^n)_{cn} = (M_c)^n$, for all $n$. Therefore 
the $(c,1)$-diagonal subalgebras
$A[{N}t]_{\Delta_{(c,1)}}$  and 
$A[{M}t]_{\Delta_{(c,1)}}$ 
are standard graded algebras over the field $k$. 
Since $c> \max\{d_M, d_N\}$ there is $n_0$ and polynomials $P_M(X, Y)$ and $P_N(X, Y)$ as in Remark~\ref{r2}
such that for all $n\gg 0$ 
$$\ell_k((M^n)_{cn}) = P_M(cn, n)+Q_M(cn, n)\quad\mbox{and}\quad
\ell_k((N^n)_{cn}) = P_M(cn, n).$$ 
In particular both the diagonal subalgebras are of Krull dimension $d+e-2$ and now for $e_1(-, -)$ defined as in the Corollary~\ref{SUV}
\begin{multline*}
e_1(A[{M}t]_{\Delta_{(c,1)}}, A[{N}t]_{\Delta_{(c,1)}})
 = \lim_{n\to\infty} \displaystyle{
\frac{\ell_k((M_c)^n/(N_c)^n)}{n^{d+e-2}/(d+e-2)!} }= \frac{1}{(d+e-1)}  (f_M(c)-f_N(c)) = 0.
\end{multline*}
Therefore 
$A[Nt]_{\Delta(c,1)}\longto 
A[{M}t]_{\Delta(c,1)}$ is an integral extension of rings which proves the assertion~(2). Further there 
 exists  $n\geq 1$ such that 
 $$(M_c)^n = (N_c)(M_c)^{n-1}\implies 
 (M_cA)^n = (N_cA)(M_cA)^{n-1}
 .$$
Now applying Theorem~\ref{t4} to the $\N$-graded $A$-modules
$N_cA = N_{\geq c}\subseteq M_cA = M_{\geq c}$ we get
${\tilde f_{M_cA}} \equiv {\tilde f_{N_cA}}$.
We note that each given  $n\geq 0$
$$
 \ell_k(M^n/(M_cA)^n) \leq \sum_{m=0}^{cn} \ell_k((M^n)_m) < \infty\quad\mbox{and}\quad
\ell_k(N^n/(N_cA)^n)\leq \sum_{m=0}^{cn} \ell_k((N^n)_m) < \infty$$ which implies
${\tilde M^n} = \tilde{(M_cA)^n}$
and
${\tilde N^n} = \tilde{(N_cA)^n}$.
Therefore
${\tilde f_{M}} \equiv {\tilde f_{M_cA}}
\equiv
{\tilde f_{N_cA}} \equiv {\tilde f_{N}}$ which proves assertion~(1).
\end{proof}

 We 
define $S=A[y]$, where $y$ is an indeterminate with $\deg y=1$, and $\mathbf{n} = \mathbf{m} + (y)$ be the unique homogeneous maximal ideal of $S$. Let $\mathsf{M} = M\tensor_AA[y]$ be the extension of the module $M$ in $S$. 

Now $S$ is a ring of dimension $d+1$ and $\rank\;{\mathsf M} = \rank\;{M} = e$ as $S$-module with an embeddding 
${\mathsf M}\longto {\mathsf F 
}$ induced by the embedding $M
\longto F$.
Note that this is standard bigraded algebras over the standard graded ring $S$, where degree of $y$ in $S = R[y]$ is $(1,0)$ and therefore $(m,n)^{th}$ component for ${\mathsf M}$ (similarly for ${\mathsf N}$) is given by 
$$({\mathsf M}^n)_m = (M^n)_m + (M^n)_{m-1}y +
(M^n)_{m-2}y^2+\cdots (M^n)_{0}y^m.$$
 
We know that  $S[{\bf M}t]$ is an integral domain of Krull dimension $d+e$.  
By Theorem~\ref{t2} the adic density function and the saturated density function is given by  
$$f_{\mathsf M}:\R\longto \R_{\geq 0}\quad \mbox{where}\quad x\to 
\lim_{n\to \infty}\frac{\ell_k(({\mathsf M}^n)_{\lfloor xn\rfloor})}{n^{d+e-1}/ (d+e)! }$$
and 
$${\tilde f}_{\mathsf M}:\R\longto \R_{\geq 0}\quad \mbox{where}\quad x\to 
\lim_{n\to \infty}\frac{\ell_k(({\mathsf M}^n:_{{\mathsf F^n}}{\bf n}^{\infty})_{\lfloor xn\rfloor})}{n^{d+e-1}/(d+e)!}$$
are  well defined functions.

We note that 
$$\frac{f_{\mathsf M}(x)}{d+e} = 
\lim_{n\to \infty}\sum_{m=0}^{\lfloor xn\rfloor}\frac{\ell_k((M^n)_m)}{n^{d+e-2}/(d+e-1)! } = 
\lim_{n\to \infty}\int_0^xf_n(y)dy  = \int_0^x\lim_{n\to \infty}f_n(y)dy 
=\int_0^xf_M(y)dy,$$
where the second last equality holds by 
 applying Lebesgue dominated convergence theorem to the pointwise converging sequence 
$\{f_n\}_n$ on the interval $[0, x)$.

For an integer $c > d_M = d_{\mathsf M}$  the diagonal $(c,1)$-subalgebra of $S[{\mathsf M}t]$, namely, 
$$S[{\mathsf M}t]_{\Delta_{(c,1)}}  = \oOplus_{n\geq 0} {({\mathsf M}^{n})}_{cn}t^{n} = \oOplus_{n\geq 0} {({\mathsf M}_c)^{n}}t^{n} $$ is standard graded over $S_0= k$ and 
$f_{\mathsf M}(c) = e(S[{\mathsf M}t]_{\Delta(c,1)})$ the multiplicity of the ring
$S[{\mathsf M}t]_{\Delta(c,1)}$.

\begin{thm}For a given $N\subseteq 
M\subseteq F$ as ~\ref{n1}, following statements are equivalent.
\begin{enumerate}
 \item $N$ is a reduction of $M$.
 \item $\rank\;M = \rank\;N$, $\varepsilon(M) = \varepsilon(N)$ and $e_i(A[Mt]) = e_i(A[Nt])$ for all $0\leq i \leq d-1$.
 \item $\rank\;M = \rank\;N$ and 
 $e(S[{\mathsf M}t]_{\Delta(c,1)}) =  
e(S[{\mathsf N}t]_{\Delta(c,1)})$ for 
some(all) integer(s) $c > \max\{d_M, d_N\}$
 \end{enumerate}
\end{thm}
\begin{proof}(1) $\implies $ (2). By Theorem~\ref{t4} we have $\rank\;M = \rank\;N$, and 
$f_M \equiv f_N$ and 
${\tilde f_M} \equiv {\tilde f_N}$. Now the rest of the assertion~(2)
follows from Theorem~\ref{l1} and Remark~\ref{r2}.

 \vspace{5pt}
 
(2)$\implies$ (3). 
Since $\rank\;M = \rank\;N = e$ we define
\begin{equation}\label{adic}f_{N}(x)= \lim_{n\to \infty}\frac{\ell_k(( N^n)_{\lfloor xn\rfloor})}{n^{d+e-2}/ (d+e-1)! }\quad\mbox{and}\quad
{\tilde  f_N}(x)= \lim_{n\to \infty}\frac{\ell_k((N^n)_{\lfloor xn\rfloor}
)}{n^{d+e-2}/(d+e-1)!}.\end{equation}

The equalities  $e_i(A[Mt]) = e_i(A[Nt])$ imply that 
${\tilde f_M} \equiv {\tilde f_N}$ and therefore 
 $f_M(x) = f_N(x)$ for  $x>\max\{d_M, d_N\}$ and we have 
 $$ 0 = \varepsilon(M) - \varepsilon(N)
 = \int_0^cf_{N}(x)dx - \int_0^cf_M(x)dx =  
 e(S[{\mathsf N}t]_{\Delta(c,1)}) -
e(S[{\mathsf M}t]_{\Delta(c,1)}),$$
where the second equality follows by Theorem~\ref{l1}.

 \vspace{5pt}
 
(3)$\implies$ (1). Let $c>\max\{d_M, d_N\}$ be an integer. Again since 
$\rank\;M = \rank\;N = e$ the adic and saturated density function for ${N}$ is given as in (\ref{adic}). Also
$${\tilde f}_{\mathsf M}(c) = f_{\mathsf M}(c)\quad\mbox{and}\quad 
{\tilde f}_{\mathsf N}(c) = f_{\mathsf N}(c).$$
But then 
$${\tilde f}_{\mathsf M}(c) = e(S[{\mathsf M}t]_{\Delta(c,1)}) =e(S[{\mathsf N}t]_{\Delta(c,1)}) =
{\tilde f}_{\mathsf N}(c).$$ 
Hece by Proposition~\ref{p1}
the canonical map $S[{\mathsf N}t]_{\Delta(c,1)}\longto
 S[{\mathsf M}t]_{\Delta(c,1)}$ is an integral extension of rings.
 Therefore there 
 exists  $m\geq 1$ such that 
 $$({\mathsf M}^m)_{cm} = {\mathsf N}_c({\mathsf M}^{m-1})_{c(m-1)}.$$
 Now, by definition
$$({\mathsf M}^m)_{cm} = (M^m)_{cm} + (M^m)_{cm-1}y +
(M^m)_{cm-2}y^2+\cdots + (M^m)_{0}y^{cm},$$
$$({\mathsf M}^{m-1})_{c(m-1)} = (M^{m-1})_{c(m-1)} + (M^{m-1})_{c(m-1)-1}y +
(M^{m-1})_{c(m-1)-2}y^2+\cdots + (M^{m-1})_{0}y^{c(m-1)},$$
$$({\mathsf N})_{c} = N_c + N_{c-1}y +
N_{c-2}y^2+\cdots + N_{0}y^{.c}.$$

Now it is enough to prove the following.

\noindent{\bf Claim}. $M^m = N.M^{m-1}$.

\vspace{5pt}
\noindent{Proof of the claim}:\quad 
 Let $x\in M^m$ be a homogeneous element.

\vspace{5pt}

\noindent{\underline {Case}}~(1).\quad $\deg~x = d_x\geq cm$. Then $x = \sum_ir_im^{i_1}\cdots m^{i_s}$, where $r_i\in A$ is   homogeneous element  and $\sum_ji_j = m$. Now $\deg~m_i\leq c$, gives that $\deg~m_1^{i_1}\cdots m_s^{i_s}\leq cm$.
Now $A$ being standard graded ring over $A_0$, we can write $r_i = r_i'r_i''$, where $r_i'$ is an element of degere $cm-\deg~m_1^{i_1}\cdots m_s^{i_s}$. Therefore 
$$ x = \sum_i\big(r'_im_1^{i_1}\cdots m_s^{i_s}\big)r_i'' \in A_{d_x-cm}(M^m)_{cm} = N_c(M^{m-1})_{c(m-1)}A_{d_x-cm}\subseteq NM^{m-1}. $$
 
\vspace{5pt}

\noindent{\underline {Case}}~(2).\quad $\deg~x = d_x < cm$. Then $x\in (M^m)_{cm-i}$, where $0< i \leq cm$.

By the above expression, comparing the coefficients of $y^i$ we get  have
$$(M^m)_{cm-i}\subseteq N_c(M^{m-1})_{c(m-1)-i}+
N_{c-1}(M^{m-1})_{c(m-1)-i+1}+\cdots +
N_{0}(M^{m-1})_{c(m-1)-i+c}.$$
 
 Therefore $M^m\subseteq NM^{m-1}$.
\end{proof}

Proof for the following theorem is similar to the one given \cite[Theorem~5.5]{DRT26}.

\begin{thm}\label{ID}Let  $N\subseteq 
M\subseteq F$ as  in~\ref{n1}, and suppose they have the same rank. Further let $S = A[y]$, ${\mathsf M} = M\tensor_AS$ and 
${\mathsf M} = M\tensor_AS$. We denote  $d_{N, M} = \max\{d_M, d_N\}$.
Then the following statements are equivalent.
\begin{enumerate}
 \item $N$ is a reduction of $M$.
 \item $e(S[{\mathsf M}t]_{\Delta_{(c,1)}}) =
 e(S[{\mathsf N}t]_{\Delta_{(c,1)}})$ for some integer $c > d_{N, M}$. 
 \item $\varepsilon(M) = \varepsilon(N)$ and $e(A[{M}t]_{\Delta(c,1)}) =  
e(A[{N}t]_{\Delta(c,1)})$ for 
some(all) integer(s) $c > d_{N, M}$.
\item  
$\varepsilon({\mathsf M}_cS) = 
 \varepsilon({\mathsf N}_cS)$ for 
some(all) integer(s) $c >  d_{N, M}$.
 \item $f_M\equiv f_N$ and 
 ${\tilde f}_M \equiv {\tilde f}_N$.
 \item 
 $e_i(S[{\mathsf M}t]) = e_i(S[{\mathsf N}t])$ for all $0\leq i \leq d$.
 \end{enumerate}
\end{thm}

\bibliographystyle{alpha}

\end{document}